\documentclass[12pt]{article}%

\usepackage{graphicx,color,epsfig}
\usepackage{amsmath,amsthm,amsfonts,amssymb}
\usepackage{epstopdf}

\providecommand{\U}[1]{\protect\rule{.1in}{.1in}}

\newcommand{\N}{\mathbb N}
\newcommand{\Q}{\mathbb Q}
\newcommand{\R}{\mathbb R}
\newcommand{\Z}{\mathbb Z}
\newcommand{\eps}{\varepsilon}
\newcommand{\Fi}{\varphi}

\newcommand{\om}{\omega}
\newcommand{\omc}{\widehat{\omega}}
\newcommand{\utheta}{\underline{\theta}}
\newcommand{\ux}{\underline{x}}

\DeclareMathOperator{\card}{card}

\DeclareMathOperator{\diam}{diam}
\DeclareMathOperator{\dist}{dist}

\DeclareMathOperator{\vol}{vol}
\DeclareMathOperator{\Sing}{Sing}
\DeclareMathOperator{\SSing}{\underline{Sing}}

\newtheorem{theorem}{Theorem}

\newtheorem{corollary}[theorem]{Corollary}

\newtheorem{question}{Question}
\newtheorem{problem}[question]{Problem}

\newtheorem{lemma}[theorem]{Lemma}

\newtheorem{proposition}[theorem]{Proposition}

\newtheorem*{theorem*}{Theorem}
\newtheorem*{theoremA*}{Theorem A}
\newtheorem*{theoremB*}{Theorem B}
\newtheorem*{theoremC*}{Theorem C}

\theoremstyle{definition}
\newtheorem{definition}{Definition}

\theoremstyle{remark}

\newtheorem{remark}{Remark}[section]

\begin{document}

\title{Hausdorff dimension and uniform exponents in dimension two}
\author{Yann Bugeaud, 
Yitwah Cheung\thanks{The second author is partially supported by NSF Grant DMS 1600476.} and Nicolas Chevallier}
\maketitle

\AtEndDocument{\noindent
  \textit{Yann Bugeaud \\
  IRMA   U.M.R. 7501 \\ 
  Universit\'e de  Strasbourg et C.N.R.S.\\ 
  7 rue Ren\' e Descartes, 67084 Strasbourg, France} \medskip \\  
  \textit{E-mail address} \texttt{bugeaud@math.unistra.fr} \bigskip \\
  \textit{ Yitwah Cheung \\
   San Francisco State University \\ 
  1600 Holloway Ave, San Francisco, CA 94132, U.S.A.} \medskip \\
  \textit{E-mail address} \texttt{ycheung@sfsu.edu} \bigskip \\
  \textit{Nicolas Chevallier \\ Haute Alsace University \\
  4 Rue des Fr\`eres Lumi\`ere, 68093 Mulhouse Cedex, France} \medskip \\
  \textit{E-mail address}  \texttt{nicolas.chevallier@uha.fr}
}


\begin{abstract}
In this paper we prove the Hausdorff dimension of the set of (nondegenerate) 
singular two-dimensional vectors with uniform exponent $\mu\in(1/2,1)$ is $2(1-\mu)$ 
when $\mu\ge\sqrt2/2$, whereas for $\mu<\sqrt2/2$ it is greater than $2(1-\mu)$ 
and at most $(3-2\mu)(1-\mu)/(1+\mu+\mu^2)$.  
We also establish that this dimension tends to $4/3$ 
(which is the dimension of the set of singular two-dimensional vectors) when 
$\mu$ tends to $1/2$.  
These results improve upon 
previous estimates of R.~Baker, joint work of the first author with M.~Laurent, 
and unpublished work of M.~Laurent.  We also prove a lower bound on the 
packing dimension that is strictly greater than the Hausdorff dimension 
for $\mu\ge0.565\dots$.  
\end{abstract}


\section{Introduction and results}
\subsection{Overview of known results}

Let  $\utheta$ be a  (column) vector  in $\R^n$. We denote by $\vert
\utheta\vert_{\infty}$
the  maximum of the   absolute values of its coordinates  and by
$$
\Vert \utheta\Vert = \min_{\ux \in \Z^n} \vert \utheta -\ux\vert_{\infty}
$$
the maximum of the distances of its  coordinates to the  rational
integers.

Let $m, n$ be positive integers and $A$ a real $n \times m$ matrix.
Dirichlet's Theorem implies that, for any $X > 1$,
the system of inequalities
$$
\Vert A\ux \Vert \le  X^{-m/n},
\quad  0 < \vert \ux \vert_{\infty} \le X  
$$
have a solution  $\ux$ in $\Z^m$. 
This leads to the following definitions. The second one was
introduced by Davenport and Schmidt \cite{DaSc70a}.

\begin{definition} 
Let $m, n$ be positive integers and $A$ a real $n \times m$ matrix. 
The matrix $A$ is badly approximable if there exists
a positive constant $c$ such that the system of inequalities    
$$
\Vert A\ux \Vert \le c \, X^{-m/n},
\quad  \quad 0 < \vert \ux \vert_{\infty} \le X   \eqno (1.1)
$$
has no solution  $\ux$ in
$\Z^m$ for any $X \ge 1$.
\end{definition}

\begin{definition} 
Let $m, n$ be positive integers and $A$ a real $n \times m$ matrix. 
We say that Dirichlet's Theorem can be improved for the matrix $A$ if there exists
a positive constant $c < 1$ such that the system of 
inequalities (1.1) has a solution  $\ux$ in
$\Z^m$ for any sufficiently large $X$.
\end{definition}

If  the subgroup $G= A \Z^m + \Z^n$ of $\R^n$ generated
by the $m$ rows of the matrix ${}^t A$ (here and below, ${}^t M$ denotes the transpose
of a matrix $M$) together with $\Z^n$ 
has rank strictly less than $m+n$, then there exists $\ux$ in $\Z^m$ 
with $|\ux|_{\infty}$ arbitrarily large, such that $\Vert A \ux \Vert = 0$ and,
consequently, for any real number $w$ and any
sufficiently large $X > 1$, the system of inequalities
$$
\Vert A\ux \Vert \le  X^{-w},
\quad  0 < \vert \ux \vert_{\infty} \le X    
$$
has a solution  $\ux$ in $\Z^m$. 
In several of the questions considered below, we have to exclude this degenerate situation, thus
we are led to introduce the set ${\cal M}^{\ast}_{n, m} (\R)$ of $n \times m$ matrices for which the
associated subgroup $G$ has rank $m+n$.

When $m=n=1$, that is, when $A = (\xi)$ for some 
irrational real number $\xi$, it is not difficult to show that Dirichlet's Theorem can be
improved if, and only if, $\xi$ is badly approximable (or, equivalently, $\xi$
has bounded partial quotients in its
continued fraction expansion); see \cite{Kh26b} and 
\cite{DaSc70a} for a precise statement. 
Furthermore, by using the theory of continued fractions, one can prove that, for any 
irrational real number $\xi$, there are arbitrarily large integers $X$ such that
the system of inequalities
$$
\Vert x \xi || \le {1 \over 2X} \quad {\rm and} \quad 0 < x \le X    \eqno (1.2)
$$
has no integer solutions; see Proposition 2.2.4 of \cite{Bu16}.

Since the set of badly approximable numbers has Lebesgue measure zero 
and Hausdorff dimension 1, this implies that the set of
$1\times 1$ matrices $A$ for which Dirichlet's Theorem can be 
improved has Lebesgue measure zero and Hausdorff dimension 1. 
The latter assertion has been extended  as follows.

\begin{theoremA*}
For any positive integers $m, n$, the set of real $n\times m$ 
matrices for which Dirichlet's Theorem
can be improved has $mn$-dimensional Lebesgue measure zero 
and Hausdorff dimension $mn$. 
\end{theoremA*}  

The first assertion of Theorem A has been established by 
Davenport and Schmidt \cite{DaSc70b} when $\min\{m, n\}= 1$.   
According to Kleinbock and Weiss \cite{KlWe}, their proof can be  
generalized to $n \times m$ matrices. Actually, a more general  
result is proved in \cite{KlWe}.    

As for the latter assertion of Theorem A, 
Davenport and Schmidt \cite{DaSc70a} showed that, for $(m, n) = (1, 2)$ or $(2, 1)$,
Dirichlet's Theorem can be improved for the $n \times m$ matrix $A$ 
if $A$ is badly approximable. They noted on page 117 that this assertion is true 
for arbitrary integers $m, n$. Combined with a result of Schmidt \cite{Schm69} 
on the size of the set of badly approximable matrices, this gives
the latter assertion of Theorem A.

We introduce the related notion of singular and regular matrices, which goes back to
Khintchine \cite{KhB}. 

\begin{definition} 
Let $m, n$ be positive integers and $A$ a real $n \times m$ matrix. 
We say that the matrix $A$ is singular if, for every positive real number $c$, the system of 
inequalities (1.1) has a solution  $\ux$ in $\Z^m$ for any sufficiently large $X$.
A matrix which is not singular is called regular.
\end{definition} 

Khintchine \cite{KhB} proved that the set of singular $n\times m$ matrices $A$ 
has  $mn$-dimensional Lebesgue measure zero; 
see also \cite{Cas}, page 92. 

A natural question is then to determine the Hausdorff dimension of the set 
of singular $n\times m$ real matrices $A$. 
The case $n=m=1$ is easy: there is no irrational real number $\xi$ such that 
the matrix $(\xi)$ is singular (recall that (1.2) has no integer solutions for arbitrarily
large values of $X$). The case $n = 2, m = 1$ was recently solved by 
Cheung \cite{cheung2}.
For an integer $n \ge 2$, we often use the terminology 
$n$-dimensional (column) vector instead of $n \times 1$ matrix. 

\begin{theoremB*}  
The Hausdorff dimension of the set of singular two-dimensional vectors
is equal to ${4 \over 3}$.
\end{theoremB*}

Cheung's result was very recently extended to $n$-dimensional vectors, for 
an arbitrary integer $n \ge 2$, by Cheung and Chevallier \cite{cheche}.

\begin{theoremC*} 
For every integer $n \ge 2$, the Hausdorff dimension of the set of singular 
$n$-dimensional vectors 
is equal to ${n^2 \over n+1}$.
\end{theoremC*} 

However, the following question remains unsolved.

\begin{problem} 
Let $m, n$ be integers at least equal to $2$.
What is the Hausdorff dimension of the set of singular 
$n \times m$ matrices ?
\end{problem}

Kadyrov {\it et al.} \cite{KKLM} established that 
this dimension is bounded from above by ${mn(m+n-1) \over m+n}$ and 
it is conjectured that there is in fact equality.

We can further discriminate between the singular matrices by introducing 
exponents of {\it uniform} Diophantine approximation. 
We keep the notation from \cite{BuLa05a}.


\begin{definition} 
Let $n$ and $m$ be positive integers
and let $A$ be a real $n\times m$ matrix.
We denote by  $\omc_{n, m}(A)$ the supremum of the real numbers $w$ 
for which, {\it for all sufficiently large} positive real numbers $X$, the system of inequalities
$$
\Vert A\ux \Vert \le X^{-w}, 
\quad  0 < \vert \ux \vert_{\infty} \le X   \eqno{(1.3)}
$$
has a solution  $\ux$ in $\Z^m$.
\end{definition} 

For $\omega$ in $(0 , + \infty]$,
let $\Sing_{n,m}(\omega)$ (resp. $\Sing^{\ast}_{n,m}(\omega)$) denote the set of  
matrices  $A$ in ${\cal M}_{n, m} (\R)$ (resp., in ${\cal M}^{\ast}_{n, m} (\R)$) such that
\[
\omc_{n,m}(A)\geq \omega,
\]
and $\SSing_{n,m}(\omega)$ (resp. $\SSing^{\ast}_{n,m}(\omega)$) denote the set of  
matrices  $A$ in ${\cal M}_{n, m} (\R)$ (resp., in ${\cal M}^{\ast}_{n, m} (\R)$) 
such that (1.3) holds for sufficiently large real numbers $X$.  Observe that the 
set $\SSing_{n,m}(\omega)$ is included in $\Sing_{n,m}(\omega)$ 
and depends on the choice of the norms 
on $\R^n$ and $\R^m$ whereas $\Sing_{n,m}(\omega)$ does not.

For a real $n \times m$ matrix $A$, Dirichlet's Theorem implies that
$$
\omc_{n, m}(A) \ge {m\over n}.   \eqno (1.4)
$$
Furthermore, we have equality in (1.4) for almost all
matrices $A$, with respect to the Lebesgue measure on $\R^{mn}$, as
follows from the Borel--Cantelli Lemma.
Any real matrix $A$ satisfying $\omc_{n, m}(A) > {m \over n}$ is singular, 
and there exist singular matrices $A$ with $\omc_{n, m}(A) = {m \over n}$. 

Since, for any real irrational number $\xi$, there are arbitrarily large
integers $X$ for which the system of inequalities (1.2) has no solutions, we deduce that,
for any $n \ge 1$, any real $n \times 1$ matrix $A$ satisfies $\omc_{n, 1}(A) \le 1$. 
Khintchine \cite{Kh26b} established that, for any integer $n \ge 2$, there exist matrices $A$
such that $\omc_{n, 1}(A) = 1$ and, for any integer $m \ge 2$ and any integer $n \ge 1$, 
there exist matrices $A$ such that $\omc_{n, m}(A) = + \infty$. 

The following problem complements Problem 1.
It has been considered by 
R. C. Baker \cite{Bak77,Bak92}, Yavid \cite{Ya87}, and Rynne \cite{Ry90a,Ry90b}. 

\begin{problem}
Let $m, n$ be positive integers.
Let $\om$ be in $[{m \over n}, + \infty]$ with $\om \le 1$ if $m = 1$.
What is the Hausdorff dimension of the set of $n \times m$ matrices $A$ 
in ${\cal M}^{\ast}_{n, m} (\R)$ satisfying $\omc_{n, m} (A) \ge \om$ 
(resp. $\omc_{n, m} (A) = \om$)? 
\end{problem}

Before stating our new results, which deal with the case $(n, m) = (2, 1)$, 
we summarize what is known towards the resolution of Problem 2. 

We first point out a result of Jarn\'\i k \cite{Jar38} asserting that
any real $1\times 2$ matrix $A$ in ${\cal M}^{\ast}_{1, 2} (\R)$ satisfies 
$$
\omc_{2, 1}({}^t A) = 1 - {1 \over \omc_{1, 2} (A)}.  \eqno (1.5)
$$
Thus, the cases $(n, m) = (1, 2)$ and $(n, m) = (2, 1)$ are equivalent. 



Let $\tau > 2$ be a real number. 
Baker \cite{Bak77,Bak92} proved that
$$
{2 \over \tau} \le
\dim_H \Sing^{\ast}_{1,2}(\tau) \le
{6 \over \tau + 1},   \eqno (1.6) 
$$
thus
$$
\dim_H  \Sing^{\ast}_{1,2}(+ \infty) = 0. 
$$
Bugeaud and Laurent \cite{BuLa05b} observed that a direct combination of 
(1.5) with a result of Dodson \cite{Do92} yields the 
slightly sharper upper bound
$$
\dim_H \Sing^{\ast}_{1,2}(\tau) \le
{3 \tau \over \tau^2 - \tau + 1},     \eqno (1.7)
$$
which was improved to $(2 \tau + 2) / (\tau^2 - \tau + 1)$ by Laurent 
in an unpublished manuscript.
We deduce from (1.5) that (1.6) and (1.7) 
give, for $\mu \ge 1/2$, 
$$
2 (1 - \mu) \le \dim_H \Sing^{\ast}_{2,1}(\mu) \le
{3 (1 - \mu) \over \mu^2 -  \mu + 1}.      \eqno (1.8)
$$
Observe that for $\mu = 1/2$ the right hand-side of (1.8) 
is equal to $2$, while Theorem B implies that
$\lim_{\mu\rightarrow\frac12}\dim \Sing^{\ast}_{2,1}(\mu) \le 4/3$. 
This shows that the right hand inequality in (1.8) is certainly not best possible
for $\mu > (\sqrt{105} - 5)/8 = 0.655 \ldots$

For $m \ge 3$, combining results of Baker \cite{Bak92} and Rynne \cite{Ry90b}, one gets that
$$
m - 2 + {m \over \tau} \le
\dim_H \Sing^{\ast}_{1,m}(\tau) \le
m - 2 + 2 {m + 1 \over \tau + 1}
$$
holds for any real number $\tau >m$, thus
$$
\dim_H \Sing^{\ast}_{1,m}(+ \infty) = m-2,
$$
for $m \ge 2$. 

As far as we are aware, there is no contribution towards Problem 2
when $\min\{m, n\} \ge 2$. 


\subsection{New results}

The purpose of the present paper is to address Problem 2
for the pair $(n, m) = (2, 1)$. 
Our first result improves the right hand inequality in (1.8) for every 
value of $\mu$ in $(1/2, 1)$. 

\begin{theorem}
	\label{thm:upper}
For any real number $\mu$ in $(1/2, \sqrt{2}/2]$, we have
$$
\dim_H \Sing^{\ast}_{2,1}(\mu) \le
{(3 - 2 \mu) \, (1 - \mu) \over \mu^2 - \mu + 1}.  
$$
For any real number $\mu$ in $[\sqrt{2}/2, 1)$, we have
$$
\dim_H \Sing^{\ast}_{2,1}(\mu) \le 2 (1 - \mu).
$$
\end{theorem}

Observe that our upper bound for $\dim_H \Sing^{\ast}_{2,1}(\mu)$ 
is a continuous function of $\mu$ in $(1/2, 1)$.

Combined with (1.8), Theorem \ref{thm:upper} yields the exact value of the dimension 
when $\mu$ is sufficiently large. 

\begin{corollary}
	\label{cor:baker}
For any real number $\mu$ in $[\sqrt{2}/2, 1)$, we have
$$
\dim_H \Sing^{\ast}_{\,2,1}(\mu) = 2 (1 - \mu).
$$
\end{corollary}

Our second result improves the left hand inequality in (1.8) for every 
value of $\mu$ in $(\sqrt{2}/2, 1)$. 

\begin{theorem}
\label{thm:lower-general}
For any real number $\mu$ in $(1/2, \sqrt{2}/2)$, we have
$$
\dim_H \SSing^{\ast}_{\,2,1}(\mu) \ge
(1 - \mu) \sup_{b>0}\frac{\left(
2b^{2}+2b\mu+b+(2-\mu)(2\mu-1\right)  )}{\left(  b+2\mu-1\right)  \left(\mu^{2}-\mu+b+1\right)  },
$$
and thus
$$
\dim_H \SSing^{\ast}_{\,2,1}(\mu) \geq 2 (1 - \mu). 
$$
\end{theorem}

A combination of Theorems \ref{thm:upper} and \ref{thm:lower-general} 
yields the following corollary.

\begin{corollary}
\label{cor:4/3}
We have
$$
\lim_{\mu \to 1/2,\,\mu>\frac{1}{2}} \,\dim_H \SSing^{\ast}_{\,2,1}(\mu)= {4 \over 3}.
$$
\end{corollary}

By Theorem B, the set of singular two-dimensional vectors 
has dimension ${4 \over 3}$. Corollary \ref{cor:4/3}
shows that there is no jump of Hausdorff dimension.

\begin{remark}
For a fixed $\mu$ in $(\frac{1}{2},\frac{1}{\sqrt{2}})$, it is
not difficult to compute the positive real number $b_0$ giving the maximum of the
rational fraction%
\[
b \longmapsto \frac{\left(  2b^{2}+2b\mu+b+(2-\mu)(2\mu-1\right)  )}{\left(  b+2\mu
-1\right)  \left(  \mu^{2}-\mu+b+1\right)  }.
\]
It satisfies a quadratic equation. Unfortunately, the lower
bound we obtain does not match 
with the upper bound established in Theorem  \ref{thm:upper}. 
\end{remark}

\begin{remark}
For real numbers $\mu, \tau \ge 1/2$,
denote by $\SSing_{\,2,1}(\mu,\tau)$ the set of matrices $A$ 
in $\SSing_{\,2,1}(\mu)$ such that there are arbitrarily large real numbers $X$
for which the system of inequalities
$$
\Vert A\ux \Vert \le X^{-\tau}, 
\quad  0 < \vert \ux \vert_{\infty} \le X  
$$
has a solution $\ux$ in $\Z$.

The proof of Theorem \ref{thm:lower-general} 
enables us to state a more precise result, namely 
\[
\dim_{H}\SSing^{\ast}_{\,2,1}(\mu,\tau)\geq(1-\mu)\frac{\left(
2b^{2}+2b\mu+b+(2-\mu)(2\mu-1\right)  )}{\left(  b+2\mu-1\right)  \left(
\mu^{2}-\mu+b+1\right),  }%
\]
where $\tau=\frac{1}{\left(  1-\mu\right)  \left(  b+1\right)  }\left(
\mu^{2}-\mu+b+1\right)  -1$ and $b$ is any positive real number (this is 
a consequence of Lemma \ref{lem:nestedness}).
\end{remark}

\begin{remark}
It is very likely that
$$
\dim_H \{A \in {\cal M}^{\ast}_{2, 1} (\R) : \omc_{2, 1} (A) = \mu\} = 2 (1 - \mu)
$$
for every $\mu$ in $[\sqrt{2}/2, 1)$. However, this does not follow from 
our results and it seems to us that a proof would require additional ideas. 
\end{remark}

Finally, we also prove a result about the packing dimension.

\begin{proposition}
\label{prop:packing}
For every real number $\mu$ in $(\frac{1}{2},1),$ we have%
\[
\dim_P\SSing^{\ast}_{\,2,1}(\mu)
\geq \sup_{b>0}\frac{\left(  2b^{2}+2b\mu+b+(2-\mu)(2\mu-1\right)
	)}{(\mu+1+2b)\left(  b+2\mu-1\right), }
\]	
thus, in particular,
	\[
	\dim_{P}\SSing^{\ast}_{\,2,1}(\mu) > 1.
	\]	
\end{proposition}

\begin{remark} Using Theorem \ref{thm:lower-general} 
and Proposition \ref{prop:packing} and some numerical experiments it is easy to see that
\[
\dim_P \SSing^{\ast}_{\,2,1}(\mu)>\dim_H \SSing^{\ast}_{\,2,1}(\mu)
\]
for $\mu\geq 0.565 \ldots$ However Theorem \ref{thm:lower-general} and 
Proposition \ref{prop:packing} are not strong enough 
to get the strict inequality for $\mu\leq 0.565 \ldots$ 
\end{remark}

\subsection*{Sketch of the proofs}
Since the proofs deal only with the sets $\Sing_{\,2,1}(\mu)$, we
will  drop the subscript ${ \, }_{\,2,1}$ when there is no ambiguity. For
convenience, we replace column vectors by row vectors. We use also the following 
notation. We take $\theta$ in $\R^2$ and consider elements 
$x = (p, q)$ in $\Z^2 \times \Z_{\ge 1}$, where $p = (p_1, p_2)$ is a pair
of integers. Then, $\frac{p}{q}$ denotes the pair $(\frac{p_1}{q},\frac{p_2}{q})$. We also
write $|x| =  q$. 

The strategy of our proofs follows closely the one of \cite{cheche}.
As in this work, the guideline for the proofs relies on two simple results.  For each primitive 
vector $x=(p,q)$  of the lattice $\Z^3$ with $p$ in $\Z^{2}$ (we keep this notation throughout 
this paper) and $q$ in $\Z_{>0}$,
let $\lambda_{1}(x)$ denote the 
length of the shortest vector of the lattice $\Lambda_{x}=\Z^{2}+\Z\frac{p}{q}$.  
Roughly, the first result is: $\theta$ in $\R^{2}$ is in $\SSing(\mu)$ 
if and only if for $n$ large enough,
\[
\lambda_{1}(x_{n})\leq\left|  x_{n}\right|  ^{-\mu}
\]
where $x_{n}=(p_{n},q_{n})$ is the $n$-th term of the 
the sequence of best approximation vectors
of $\theta$  and $\left|  x_n\right|  = q_n$ (see   
Section \ref{sec:Farey} and Corollary \ref{cor:sing} for an exact statement). 
The second result is a multidimensional extension of Legendre's Theorem 
about convergents of ordinary continued fraction expansions: if $x=(p,q)$ 
is a best approximation vector of $\theta$, then 
$\theta\in B(\frac{p}{q},\frac{2\lambda	_{1}(x)}{\left|  x\right|  })$ 
and conversely, if $\theta\in B(\frac{p}{q},\frac{\lambda_{1}(x)}{2\left|  x\right|})$, 
then $x$ is a best approximation vector of $\theta$ (see Lemma~\ref{lem:BAI:2}).  
Then we use the standard strategy for computing the Hausdorff dimension of 
Cantor sets defined by a nested tree of intervals. Precisely, defining the children of an 
interval  as the immediate successors with respect to the partial order 
induced by inclusion  
of intervals,  the diameter of one interval raised to the 
power $s$ has to be compared with the sum over all the children intervals of their
diameters raised to the power $s$.  

For the upper bound, consider a set $\sigma_{\mu}(x)$ 
for each primitive vector $x=(p,q)$ in
$\Z^{2}\times\Z_{>0}$ with $\lambda_{1}(x)\leq\left|  x\right|^{-\mu}$.
This set plays the role of the children of $x$.  
The first idea is to take for $\sigma_{\mu}(x)$ the set of all possible
primitive vectors $y$ in $\Z^{2}\times \Z_{>0}$ with $\lambda_{1}(y)\leq\left|
y\right|  ^{-\mu}$ such that $x$ and $y$ are two consecutive
best approximation vectors of some $\theta$ in $\R^{d}$. If for all $x$,
\[
\sum_{y=(u,v)\in\sigma_{\mu}(x)}\left(\diam B\left(\frac{u}{v},\frac{2\lambda
	_{1}(y)}{\left|  y\right|  }\right)\right)^{s}\leq\left(\diam B\left(\frac{p}{q},\frac{2\lambda
	_{1}(x)}{\left|  x\right|  }\right)\right)^{s}%
\]
then the Hausdorff dimension of $\SSing(\mu)$ is at most equal to $s$. We make
this statement more precise by using self-similar covering introduced by the second
author (see \cite{cheung2} and Theorem~\ref{thm:SSCupper}). However the above
inequality does not hold and as in \cite{cheung2} and \cite{cheche} 
we modify the definition of the set
$\sigma_{\mu}(x)$ with an ``acceleration" by considering only a subsequence of 
the sequence of best approximations (see Definition \ref{def:sigma-upper}). 
Note that the subsequence is not the same as that in \cite{cheche}.  
Another point is that it is better to use a radius larger than 
$\frac{2\lambda_{1}(x)}{\left|  x\right|}$ (see Corollary \ref{cor:upper-self}), 
for it avoids the second acceleration 
used in \cite{cheung2}. The choice of a good radius is more delicate than in \cite{cheche}.  
With these ingredients the proof of the upper bound 
follows readily; see Section~\ref{sec:upper-bound}.

The lower bound is trickier. The idea is to find a Cantor set included in
$\SSing^{\ast}(\mu)$. This Cantor set has an ``inhomogeneous" tree structure.
For each $x=(p,q)$ such that $\lambda_{1}(x)\leq\left|  x\right|  ^{-\mu}$, 
we define a finite set $\sigma(x)$ and  a ball $B(x)$ 
such that for all $z=(u,v)$ in $\sigma(x)$, 
we have both $\lambda_{1}(z)\leq\left|  z\right|  ^{-\mu}$ and 
\[
B(z)\subset B \Bigl(\frac{u}{v},\frac{\lambda_{1}(z)}{2\left|  z\right|  } \Bigr) \subset B(x).
\]
The above inclusions ensure that  $x$ and $z$ are best approximation vectors of 
all $\theta$ in $B(z)$ which in turn will be helpful to show that the Cantor set defined 
by the sets  $\sigma(x)$ and the balls $B(x)$ is included in $\SSing^{\ast}(\mu)$ 
(see Proposition \ref{prop:SSS}).  Then, the inequality
\[
\sum_{z\in\sigma(x)}\left(\diam B\left(z\right)\right)^{s}\geq\left(\diam B\left(x \right)\right)^{s}
\]
together with a condition about the distribution in $B(x)$ of the points $z$ in $\sigma(x)$  
imply that the Hausdorff dimension of $\SSing^{\ast}(\mu)$ is 
at least equal to $s$.   
However, this program is not straightforward because the condition
about the distribution of the elements of $\sigma(x)$ 
used in \cite{cheche} does not work in our context 
(see Theorem 3.6 of \cite{cheche}).

To overcome this problem, we use a more flexible
condition which is an adaptation of the mass
distribution principle to self-similar covering; see Theorem \ref{thm:massdistribution}.
This more flexible condition, together with a careful study of the geometric
positions of the points of  $\sigma(x)$ in the ball $B(x)$    
(see Lemmas \ref{lem:counting} and \ref{lem:distortedtiling}), finally lead to the lower bound.

\subsection{Questions and problems}

In this subsection, we gather some suggestions for further research 
closely related to the present work. 

Maybe, it is possible to  adapt the 
methods of \cite{DaSc70b,KlWe}  
to solve the following problem, which seems to be rather difficult.  

\begin{problem}
Let $c$ be a real number with $0<c<1$.
What is the Hausdorff dimension    
of the set of $n\times m$ matrices such that (1.1) has a solution  $\ux$ 
in  $\Z^m$ for any sufficiently large $X$? Is this a continuous 
function of $c$? 
\end{problem}


All the results quoted above are concerned with
approximation of independent quantities in the sense that we assume that the entries
of the matrices $A$ are independent. It is a notorious fact that 
questions of approximation of dependent quantities are much more delicate. 
An emblematic example in the case of $n \times 1$ matrices is given by the 
Veronese curve $(\xi, \xi^2, \ldots , \xi^n)$. At present, we do not know the 
Hausdorff dimension of the set of real numbers $\xi$
such that the pair $(\xi, \xi^2)$ is singular.
In 2004 Roy \cite{Roy04} showed that this set is nonempty. 
In the oppposite direction, 
Shah \cite{Shah09,Shah10} has obtained several striking results on the size 
of sets of matrices with dependent entries for which Dirichlet's Theorem 
cannot be improved.

\begin{problem}
Let $n \ge 2$ be an integer.
What is the Hausdorff dimension    
of the set of real numbers $\xi$ such that $(\xi, \xi^2, \ldots , \xi^n)$ 
is singular? 
\end{problem}

The latter problem is deeply connected with the following famous conjecture of Wirsing on
approximation to real numbers by algebraic numbers of bounded degree. 
Recall that the height of an algebraic number $\alpha$, denoted by $H(\alpha)$, is the 
maximum of the absolute values of the coefficients of its minimal defining 
polynomial over $\Z$. 

\begin{problem} (Wirsing) 
Let $n \ge 2$ be an integer and $\xi$ be a transcendental real number. 
For any positive $\eps$, there exist algebraic numbers $\alpha$ of degree at most $n$ 
and of arbitrarily large height such that
$$
|\xi - \alpha| < H(\alpha)^{-n-1 + \eps}. 
$$
\end{problem}

It follows from results established in \cite{BuLa05a} that the Hausdorff dimension of the 
set of counterexamples to the Wirsing conjecture on the approximation by
algebraic numbers of degree at most $n$ is at most equal to the 
Hausdorff dimension of the set of real numbers $\xi$ such that $(\xi, \xi^2, \ldots , \xi^n)$ 
is singular. See Chapter 3 of \cite{BuLiv} for a survey of known results 
towards Wirsing's conjecture.

A further line of research is Diophantine approximation on fractal sets. 
Rather than assuming that $A$ is an arbitrary real $n \times m$ matrix, 
we restrict our attention to matrices in a given fractal set. 

\begin{problem}
What is the Hausdorff dimension of the set of singular pairs whose entries 
belong to the middle third Cantor set?
\end{problem}

Our results on the packing dimension motivate the following questions. 

\begin{problem}
Is the packing dimension of $\operatorname{Sing}^{\ast}(\mu) $ 
strictly greater than the Hausdorff dimension for all $\mu> 1/2$?
What is the value of the packing dimension of the set of singular pairs? Is it equal
to its Hausdorff dimension, that is, to $4/3$? 
\end{problem}

\section{Definitions and results about self-similar coverings}

\begin{definition}
\label{def:self-sim} Let $Y$ be a metric space. A \emph{self-similar
structure} on $Y$ is a triple $(J,\sigma,B)$ where $J$ is countable, $\sigma$
is a subset of $J^{2}$, and $B$ is a map from $J$ into the set of bounded
subsets of $Y$. A \emph{$\sigma$-admissible sequence} is a sequence
$(x_{n})_{n\in\N}$ in $J$ such that

\begin{enumerate}
\item[(i)] for all integers $n$, $(x_{n},x_{n+1})\in\sigma$.
\end{enumerate}

Let $X$ be a subset of $Y$. A \emph{self-similar covering} of $X$ is a
self-similar structure $(J,\sigma,B)$ such that, for all $\theta$ in $X$, there
exists a $\sigma$-admissible sequence $(x_{n})_{\in\N}$ in $J$ satisfying

\begin{enumerate}
\item[(ii)] $\lim_{n\to\infty}\diam  B(x_{n})=0$,

\item[(iii)] $\bigcap_{n\in\N} B(x_{n})=\{\theta\}$.
\end{enumerate}

The set covered by a self-similar structure $(J,\sigma,B)$ is the set all
$\theta$ in $Y$ with the two properties above.
\end{definition}

\textbf{Notation. } We denote by $\sigma(x)$ the set of $y$ in $J$ such that
$(x,y)\in\sigma$.

\begin{definition}
By a strictly nested self-similar structure we mean a self-similar structure $(J,\sigma,B)$
that satisfies
$\lim_{n\rightarrow\infty}\diam B(x_{n})=0$, for all $\sigma$-admissible
sequence $(x_{n})_{n\in\N}$, and $B(y)\subset B(x)$, for all $x$ in $J$
and all $y$ in $\sigma(x)$.
\end{definition}

\subsection{Upper bound for the Hausdorff dimension}

We quote a result from \cite{cheung2}. 

\begin{theorem}
(\cite{cheung2})\label{thm:SSCupper} Let $Y$ be a metric space, let $X$ be a
subset of $Y$ that admits a self-similar covering $(J,\sigma,B)$ and let $s$
be a positive real number. If 
$$
\sum_{y\in\sigma(x)}\diam B(y)^{s}\leq\diam %
B(x)^{s}, 
$$
holds for all $x$ in $J$, 
then $\dim_{H}X\leq s$.
\end{theorem}

\subsection{Lower bound for the Hausdorff dimension}

There already exist results providing lower bounds
for the Hausdorff dimension of self 
similar structures, see \cite{cheung2} or \cite{cheche}. However these results are not 
suitable for our purpose. An adaptation of the mass distribution principle 
to self similar structures leads to a more flexible statement.  

Let $(J,\sigma,B)$ be a self-similar
structure on a complete metric space $(Y,d)$. For a subset $F$ of $Y$ and $x$ in $J$, we set%
\[
\sigma_{F}(x)=\{y\in\sigma(x):F\cap B(y)\neq\emptyset\}.
\]

\begin{theorem}
\label{thm:massdistribution} Let $(J,\sigma,B)$ be a strictly nested self-similar
structure on a complete metric space $(Y,d)$. Suppose that, for all $x\in J$, the set 
$B(x)$ is bounded and closed. Let $s$ be a positive real number and suppose that

\begin{itemize}
\item[i.] for all $x$ in $J$, $\diam B(x)>0$ and $\sum_{y\in
\sigma(x)}(\diam B(y))^{s}\geq(\diam B(x))^{s}$,

\item[ii.] for all $x$ in $J$, the sets $B(y)$, $y\in\sigma(x)$, are disjoint,

\item[iii.] there exists a constant $C$ such that for all $x$ in $J$ and all subsets
$F$ in $Y$ such that
\[
\delta(x)=\min_{y\neq y^{\prime}\in\sigma(x)}d(B(y),B(y^{\prime}%
))\leq\diam F\leq\diam B(x),
\]
we have%
\[
\frac{\sum_{y\in\sigma_{F}(x)}(\diam B(y))^{s}}%
{(\diam F)^{s}}\leq C\frac{\sum_{y\in\sigma(x)}%
(\diam B(y))^{s}}{(\diam B(x))^{s}},
\]

\end{itemize}
Then $\dim_{H}E\geq s$ and the Hausdorff 
dimension of the set covered by $(J,\sigma,B)$ is $\geq s$.
\end{theorem}

We need an auxiliary Lemma.
Let $(J,\sigma,B)$ be a self-similar structure on a complete metric space
$(Y,d)$. For $x_{0}$ in $J$, we consider the set $\Omega_{x_{0}}$ of all
admissible sequences starting at $x_{0}$ and, for a finite admissible sequence
$a_{0}=x_{0},a_{1} \ldots ,a_{n}$ in $J$, we denote by 
\[
\lbrack a_{1}, \ldots ,a_{n}]=\{(x_{n})_{n\in\N}\in\Omega_{x_{0}}%
:x_{i}=a_{i},\ i=1, \ldots ,n\}
\]
the associated cylinder. We endow $\Omega_{x_{0}}$ with the topology induced
by the product topology on $J^{\N}$.

\begin{lemma}
	\label{lem:compactness}Let $(J,\sigma,B)$ be a strictly nested self-similar
	structure on a complete metric space $(Y,d)$. Suppose that, for all $x\in J$, the set
	$B(x)$ is bounded and closed. Then $\Omega_{x_{0}}$\textit{ is a
		compact subset of }$J^{\N}$ and for all sequence 
	$(x_{n})_{n\in\N}$ in $\Omega_{x_{0}}$ there exists
	a unique point $a$ in the intersection of the closed sets 
	$B(x_{n})$\textit{, }$n\in\N$. Furthermore the map 
	$\varphi:\Omega_{x_{0}}\rightarrow Y$ defined by $\varphi
	((x_{n})_{n\in\N})=a$ is continuous and the sequence 
	\[
	D_{n}=\max\{\diam \varphi([x_{1}, \ldots ,x_{n}]):x_{1}, \ldots ,x_{n}\in
	J\}
	\]
	goes to zero when $n$ goes to infinity.
\end{lemma}

\begin{proof}[Proof of the Lemma] The only thing which is not clear is the last point. Consider
	the sequence of functions $(d_{k})_{k\geq1}$ defined by%
	\[
	d_{k}((x_{n})_{n\in\N})=\diam \varphi([x_{1}%
	, \ldots ,x_{k}])
	\]
	for a sequence $(x_{n})_{n\in\N}$ in $\Omega_{x_{0}}$. By the definition
	of the topology, each $d_{k}\ $is continuous on the compact set $\Omega
	_{x_{0}}$. Clearly the sequence $(d_{k})_{k}$ is non-increasing and by
	assumption $\lim_{k\rightarrow\infty}d_{k}((x_{n})_{n\in\N})\leq
	\lim_{k\rightarrow\infty}\diam B(x_{k})=0$ for all
	$(x_{n})_{n\in\N}$ in $\Omega_{x_{0}}$, hence by Dini's theorem, the
	sequence $(d_{k})_{k\geq1}$ converges uniformly to zero. 
\end{proof}

\begin{proof}[Proof of Theorem \ref{thm:massdistribution}] We keep the notations of the Lemma. The set 
$E:=\varphi(\Omega_{x_0})$ is a compact
subset of $Y$.
It is enough to prove that there exists a probability measure
$\nu$ on $Y$ supported by $E$ such that for every Borel subset $F$ of $Y$, we have%
\[
\nu(F)\leq C(\diam F)^{s}%
\]
for some absolute constant $C$.

A map $\mu$ defined on the set of cylinders can be extended to a probability
measure on $\Omega_{x_{0}}$ if for all cylinders $[x_{1}, \ldots ,x_{n}]$ we have
the additive formula
\[
\sum_{x\in\sigma(x_{n})}\mu([x_{1}, \ldots ,x_{n},x])=\mu([x_{1}, \ldots ,x_{n}]).
\]
For all $x$ in $J$ set $M(x)=\sum_{y\in\sigma(x)}(\diam %
B(y))^{s}$. The following recursion formulas
\begin{align*}
\mu([x_{1}])  &  =\frac{(\diam B(x_{1}))^{s}}{M(x_{0})},\\
\mu([x_{1}, \ldots ,x_{n+1}])  &  =\frac{(\diam B(x_{n+1}))^{s}%
}{M(x_{n})}\mu([x_{1}, \ldots ,x_{n}]),
\end{align*}
define a measure $\mu$ on the set of cylinders. Clearly the additive formula
holds, hence $\mu$ extends to a probability measure.

Call $\nu$ the image of $\mu$ by the map $\varphi$. The support of $\nu$ is
included in $E$.

We want to check that $\nu(F)\leq C(\diam F)^{s}$ for all Borel
subset $F$ of $Y$. We can suppose that $F\subset E$.

First, let us show by induction that for all cylinders $[x_{1}, \ldots ,x_{n}]$, we
have the inequality,
\[
\mu([x_{1}, \ldots ,x_{n}])\leq\frac{(\diam B(x_{n}))^{s}%
}{(\diam B(x_{0}))^{s}}.
\]
For all $x_{1}\in\sigma(x_{0})$,
\begin{align*}
\mu([x_{1}])  &  =\frac{(\diam B(x_{1}))^{s}}{M(x_{0})}\\
&  \leq\frac{(\diam B(x_{1}))^{s}}{(\diam %
B(x_{0}))^{s}},
\end{align*}
and since $M(x_{n})\geq(\diam B(x_{n}))^{s}$,%
\begin{align*}
\mu([x_{1}, \ldots ,x_{n+1}])  &  =\frac{(\diam B(x_{n+1}))^{s}%
}{M(x_{n})}\mu([x_{1}, \ldots ,x_{n}])\\
&  \leq\frac{(\diam B(x_{n+1}))^{s}}{(\diam %
B(x_{n}))^{s}}\times\frac{(\diam B(x_{n}))^{s}}%
{(\diam B(x_{0}))^{s}}\\
&  \leq\frac{(\diam B(x_{n+1}))^{s}}{(\diam %
B(x_{0}))^{s}}.
\end{align*}

Let $F$ be a subset of $E$. If $F$ is reduced to one point $a=\varphi
((x_{n})_{n\in\N})$, we have to check that $\nu(F)=0$. By the
disjointness assumption $\varphi$ is one to one and
\[
\nu(F)\leq\nu(\varphi([x_{1}, \ldots ,x_{n}]))=\mu([x_{1}, \ldots ,x_{n}])\leq
\frac{(\diam B(x_{n}))^{s}}{(\diam B(x_{0}))^{s}}, 
\]
which goes to zero because the self-similar covering is strictly nested.

Suppose now that $\diam F>0$. By the last point of the above
lemma there is a cylinder $\mathcal{C=}[x_{1}, \ldots ,x=x_{n}]$ of maximal length
containing the image $\varphi(\mathcal{C})$ ($\mathcal{C}$ can be
$\Omega_{x_{0}}$). By maximality, there exists $y\neq y^{\prime}$ in
$\sigma(x)$ such that $F$ intersects both $B(y)$ and $B(y^{\prime})$, hence
$\diam F\geq\delta(x)$. Therefore, 
\[
\frac{\sum_{y\in\sigma_{F}(x)}(\diam B(y))^{s}}%
{(\diam F)^{s}}\leq C\frac{\sum_{y\in\sigma(x)}%
(\diam B(y))^{s}}{(\diam B(x))^{s}}=C\frac
{M(x)}{(\diam B(x))^{s}}. 
\]
By the definition of $\sigma_{F}$, we have%
\begin{align*}
F &  \subset(\cup_{y\in\sigma_{F}(x)}B(y)),\\
\nu(F) &  \leq\sum_{y\in\sigma_{F}(x)}\nu(B(y)), 
\end{align*}
and, by the definition of $\nu$ and by the disjointness assumption, %
\begin{align*}
\sum_{y\in\sigma_{F}(x)}\nu(B(y)) &  =\sum_{y\in\sigma_{F}(x)}\mu
([x_{1}, \ldots ,x,y])\\
&  =\sum_{y\in\sigma_{F}(x)}\mu([x_{1}, \ldots ,x])\frac{(\diam %
B(y))^{s}}{M(x)}.
\end{align*}
Hence, we deduce from the above inequality about cylinders that 
\begin{align*}
\nu(F) &  \leq\sum_{y\in\sigma_{F}(x)}\frac{(\diam B(x))^{s}%
}{(\diam B(x_{0}))^{s}}\frac{(\diam B(y))^{s}%
}{M(x)}\\
&  \leq\frac{C}{(\diam B(x_{0}))^{s}}(\diam %
F)^{s}.\ 
\end{align*}
\end{proof}

\subsection{Lower bound for the packing dimension}

\begin{lemma}
	\label{lem:self:packing}
	Let $(J,\sigma,B)$  be a strictly nested self-similar structure on a metric space $Y$ and  
	let $s$ be a positive real number. 
	Suppose that we have a map $x \mapsto \widehat{x}$ from $J$ to $Y$  
	and a map $B': x \mapsto B'(x)=B(\widehat{x},r(x))$ from $J$ to the set of closed balls in $Y$. 
	We also make the following assumptions: 
	
	\begin{enumerate}
		\item for all $x$ in $J$, $\sigma(x)$ is finite,
		\item there exists $k<1$ such that $B(x)\subset B(\widehat{x},kr(x))$ for all $x$ in $J$,
		\item for all $x$ in $J$, the balls $B'(y)$, $y\in\sigma(x)$, 
		are disjoint and included in $B'(x)$,
		\item  for all $\sigma$-admissible sequence $(x_{n})_{\in\N}$ in $J$, we have 
		$\lim_{n\to\infty}\diam B'(x_{n})=0$,
		\item for all $x$ in $J$, $\diam B'(x)>0$ and $\sum_{y\in
			\sigma(x)}(\diam B'(y))^{s}\geq(\diam B'(x))^{s}$.
		\end{enumerate}
	Then, the packing dimension of the set covered by $(J,\sigma,B)$ is at least equal to $s$.
\end{lemma}

\begin{proof} We keep the notations of the previous section and consider, for $x_{0}\in J$, 
the set $\Omega_{x_{0}}$ of all
admissible sequences starting at $x_{0}$. We are going to show that
\[
\dim_P E\geq s.
\]
Let $\varepsilon$ be a positive real number. As in the proof of Lemma \ref{lem:compactness}, 
Dini's theorem implies that
\[
\lim_{p \rightarrow \infty}\sup\{\diam B'(x_p):(x_n)_{n\in \N}\in \Omega_{x_0}\}=0.
\] 
Therefore, there exists an integer $q_{\eps}$ such that 
$$
\sup\{\diam B'(x_{q_{\eps}}):(x_n)_{n\in \N}\in \Omega_{x_0}\}\leq \varepsilon.
$$
For a positive integer $q$, let $J_q$ be the set of $x$ in $J$ such that 
there exists a $\sigma$-admissible sequence $x_0,x_1, \ldots ,x_q$ with $x_q=x$. 
The disjointness property in item 3 
implies that the sets $\sigma(x)$, $x\in J_q$, are disjoint. 
Hence, we have a disjoint union $J_{q+1}=\cup_{x \in J_q}\sigma(x)$. 
An easy induction together with item 5 implies that for all $q$,
\[
\sum_{x\in J_{q}}(\diam B'(x))^{s}\geq(\diam B'(x_0))^{s},
\]
hence we would have shown that the $\eps$-packing measure satisfies
\[
\mathcal {P}^s_{\varepsilon}(E)\geq \diam B'(x_0))^{s},
\] 
if the balls $B'(x)$, $x \in J_{q_{\eps}}$, were centered at points in $E=\varphi(\Omega_{x_0})$. 
Now, by item 2, the set 
$\Fi([x_0, \ldots ,x_q])$ is included in the ball $B'(\widehat{x}_q,kr(x_q))$, 
hence there is a point $y(x_q)\in E$ such that the ball $B(y(x_q),(1-k)r(x_q))$ 
is included in the ball $B(\widehat{x}_q,r(x_q))$. It follows that 
\[
\sum_{x\in J_{q}}(\diam B(y(x),(1-k)r(x)))^{s}\geq((1-k)\diam B'(x_0))^{s}, 
\]
which in turn implies that $\mathcal P^{s'}(E)=\infty$ for all $s'<s$.
It remains to show that the packing measure $p^{s'}(E)$ does not vanish. 
This is proved by means of a standard argument. If $(E_i)_{i\in\
N}$ is any covering of $E$, then, by Baire's Theorem, 
one of the closure $F_i=\bar{E_i}$, say $F_q$, contains a subset of $E$ 
of nonempty relative interior. It follows that there exists a cylinder $C=[a_0, \ldots ,a_j]$ 
of $\Omega_{x_0}$ such that $\Fi(C)\subset F_q$. 
Now, the previous way of reasoning implies that 
\[
\mathcal P^s(\Fi(C))\geq ((1-k)\diam B'(a_j))^s, 
\]
hence, for all $s'<s$,
\[
\mathcal P^{s'}(F_q)=\mathcal P^{s'}(E_q)=\infty
\]
and $p^{s'}(E)=\infty$. 
\end{proof}

\section{Farey Lattices and best approximants \label{sec:Farey}}

From now on we suppose that $\R^2$ is equipped with the standard Euclidean norm $\|.\|_e $.

Let the set of primitive vectors in $\Z^{3}$ corresponding to
rationals in $\Q^{2}$ in their ``lowest terms representation" be
denoted by
\[
Q=\{ (p_{1},p_{2},q)\in\Z^{3}: \gcd(p_{1},p_{2},q)=1, q>0 \}.
\]
Given $x=(p,q)\in Q$, where $p\in\Z^{2}$, we use the notation
\[
\left\vert  x \right\vert=q \quad\text{ and }\quad\widehat x=\frac{p}{q}.
\]
For $x$ in $Q$, let
\[
\Lambda_{x} := \Z^{2}+\Z\widehat x = \pi_{x}(\Z^{3})
\]
where $\pi_{x}:\R^{3}\to\R^{2}$ is the "projection along the lines
parallel to $x$" given by the formula $\pi_{x}(m,n)=m-n\widehat x$ for
$(m,n)\in\R^{2}\times\R$. Observe that $\vol \Lambda
_{x}=\left\vert  x \right\vert^{-1}$.

Given a norm on $\R^{2}$, we denote the successive minima of $\Lambda_{x}$
by $\lambda_{i}(x)$ and the \emph{normalized} successive minima by
\[
\widehat\lambda_{i}(x) := \left\vert  x \right\vert^{1/2}\lambda_{i}(x) \quad\text{ for }\quad
i=1,2.
\]

We collect without proof a few lemmas the proof of which can be found in  \cite{cheung2} and \cite{cheche}.

\subsection{Inequalities of best approximation}
The ordinary continued fraction expansion is a very efficient tool for the study of Diophantine exponents of a single real number.
In higher dimensions, it is convenient to replace the ordinary continued fraction expansion by the sequence of best Diophantine approximations vectors because a weak form of many properties of the one-dimensional expansion still hold.  

Recall that the sequence $(q_n)_{n\ge 0}$ of 
\emph{best simultaneous approximation denominators} of $\theta\in\R^2$ 
with respect to the norm $\|\cdot\|_e$ is defined by the recurrence relation 
$$
q_0=1,\quad q_{n+1}=\min\{q\in\N: q>q_n, \dist(q\theta,\Z^2)<\dist(q_n\theta,\Z^2)\}.
$$  
By definition, the sequence $(q_n)_{n\ge 0}$ is strictly increasing, 
while the sequence $(r_n)_{n\ge 0}$ 
where $r_n=\dist(q_n\theta,\Z^2)$, is strictly decreasing.  
These sequences are infinite if and only if $\theta\in\R^2\setminus\Q^2$.  
For each $n\ge 0$, we choose $p_n$ so that $\|q_n\theta-p_n\|_e=r_n$
and set $x_n=(p_n,q_n)\in\Z^2\times\Z_{>0}$. 
It is customary to refer to $(x_n)_{n\ge 0}$ as 
\emph{the sequence of best simultaneous approximation vectors}, 
even though the choice of $p_n$ need not be unique.\footnote{It is unique 
	as soon as $q_n$ is large enough, e.g. if $q_n>(4\mu_2/\lambda_1(\Z^2))^2$.  
	See \cite{lagarias1} or Remark~2.13 of \cite{cheung2}.}
See \cite{chevallier1,lagarias1,lagarias2,moshchevitin07}
for more about best approximations.  In what follows we shall often write 
best approximation instead of best simultaneous approximation vector.

First we qote a result that generalizes \emph{Legendre's} Theorem: \emph{$p/q$ is a convergent of $\alpha \in \R$ as soon as $|\alpha-p/q|<1/2q^2$.}   
Denote by $\mu_2$ the supremum of $\lambda_1(L)$ over all $2$-dimensional 
lattices $L\subset\R^2$ of covolume $1$.  

\begin{lemma}[Thm.~2.11 of \cite{cheung2}]\label{lem:BAI:2}
	For $x\in Q$, let $\Delta(x)=\{\theta:\widehat x\textrm{ is a best approximation of }\theta\}$.  
	If $\left\vert  x \right\vert>\left(\frac{\mu_2}{\lambda_1(\Z^2)}\right)^2$, then 
	$$
	\Bar {B}\left(\widehat x,\frac{\lambda_1(x)}{2\left\vert  x \right\vert}\right) \subset 
	\Delta(x) \subset B\left(\widehat x,\frac{2\lambda_1(x)}{\left\vert  x \right\vert}\right),
	$$
	where $\Bar {B}$ denote the closed ball.    
\end{lemma}

The unimodular property, $\left\vert p_{n+1}q_n-q_{n+1}p_n\right\vert=1$, which hold for two consecutive convergents $\frac{p_{n}}{q_n}$ and $\frac{p_{n+1}}{q_{n+1}}$ of the ordinary continued fraction expansion cannot be extended to best Diophantine approximations in higher dimensions (see \cite{chevallier1} and \cite{moshchevitin07}). However (i) of Lemma \ref{lem:BAI:3} can be seen as a  weak form of the unimodular property.

The notation $x \asymp_2 y$ means $\frac12y\le x\le 2y$.  

\begin{lemma}[\cite{cheung2}, \cite{cheche}]\label{lem:BAI:3}
	Let $x_n=(p_n,q_n), n\ge0, $ be the sequence of best approximation vectors of $\theta\in\R^2$.  
	Then 
	\begin{enumerate}
		\item[(i)] $\|\widehat x_n-\widehat x_{n+1}\|_e < \frac{4\lambda_1(x_{n+1})}{|x_n|}$.  
		\item[(ii)] For all $k\ge0$, 
		$\|\widehat x_n-\widehat x_{n+k}\|_e < \frac{4\lambda_1(x_n)}{|x_n|}$.  
		\item[(iii)] For all $y=(p,q)\in\Z^{2+1}$ 
		with $0<q<|x_n|$, $\|p-q\theta\|_e \asymp_2\|p-q\widehat x_n\|_e$.  
	\end{enumerate}
\end{lemma}

The previous lemma allows to almost characterize 
the set $\SSing(\mu)$ with best approximation vectors. 

\begin{corollary}
	\label{cor:sing}
	Let $\mu^{\prime}>\mu>0$ and let $\theta$ be in $\R^2$. 
	Call $x_n=(p_n,q_n), n\ge0$,  the sequence of best approximation vectors of $\theta\in\R^2$.	
	If $\theta\in \SSing(\mu^{\prime})$, then for all $n$ large enough
	\[
	\lambda_1(x_n)\leq \left\Vert q_{n-1}\widehat{x}_{n}-p_{n-1}\right\Vert_e
	\leq \left\vert x_{n} \right\vert^{-\mu}.
	\]
	Conversely, if
	\[
	\lambda_1(x_n)\leq 
	\left\vert x_{n} \right\vert^{-\mu^{\prime}}
	\] 
	for all $n$ large enough, then $\theta\in \SSing(\mu)$.
\end{corollary}

\begin{proof}
	By Lemma \ref{lem:BAI:3} (iii), if $\theta\in \SSing(\mu^{\prime})$, then for all $n$ large enough
	\begin{align*}
	\lambda_1(x_n)&\leq \left\Vert q_{n-1}\widehat{x}_{n}-p_{n-1}\right\Vert_e\\
	&\leq 2\left\Vert q_{n-1}\theta-p_{n-1}\right\Vert_e \\
	&\leq 2(q_n-1)^{-\mu^{\prime}} \\
	&\leq \left\vert x_{n} \right\vert^{-\mu}.
	\end{align*}
	Conversely, if 
	$
	\lambda_1(x_n)\leq 
	\left\vert x_{n} \right\vert^{-\mu^{\prime}}
	$, then  by Lemma \ref{lem:BAI:3} (iii) and (i), for all $q_{n-1}\leq q<q_{n}$, we have
	\begin{align*}
	d(\{\theta, \ldots ,q\theta\},\Z^{2})&=\left\Vert q_{n-1}\theta-p_{n-1}\right\Vert_e\\
	&\leq 2\left\Vert q_{n-1}\widehat{x}_{n}-p_{n-1}\right\Vert_e\\
	&\leq 8\lambda_1(x_n)\\
	&\leq 8q_n^{-\mu^{\prime}}
	\leq q^{-\mu}, 
	\end{align*} 
	when $n$ is large enough.
\end{proof}

\subsection{The subspace $H_x$ }\label{S:hyperplane}
Call $x_n=(p_n,q_n)$, $n\in \N$, the sequence of best approximation vectors of $\theta \in \R^2$.
Corollary \ref{cor:sing} shows that if  $\theta$ is in $\SSing(\mu)$ with $\mu>\frac12$, then $\widehat{\lambda}(x_n)\rightarrow 0$ when $n$ goes to $\infty$. 
It follows that the shortest vector of the lattice $\Lambda_{x_n}$ is very small compare to $\lambda_2(x_n)$ when $n$ is large.
So, at the scale of the second minimum, the lattice $\Lambda_{x_n}$ looks like an evenly spaced union of lines parallel to the shortest vector,  with very closed points evenly spaced in these lines.  
This picture is helpful  and
shows that the line defined by the shortest vector  should play an important role.
The subspace $H_x$ defined below could have been defined with the shortest vector of the lattice $\Lambda_x$. However as  in \cite{cheche} we use the volume instead of  the length because it works in any dimension. \medskip

For each $x$ in $Q$ we fix once and for all a co-dimension one sub-lattice of
$\Lambda_{x}$ of minimal volume and call it $\Lambda^{\prime}_{x}$. Let
$H_{x}=\pi_{x}^{-1}H^{\prime}_{x}$ where $H^{\prime}_{x}$ is the real span of
$\Lambda^{\prime}_{x}$. Thus,
\[
\Lambda^{\prime}_{x} = \Lambda_{x} \cap H^{\prime}_{x}.
\]

The two Lemmas below are easy and proved in \cite{cheche}.

\begin{lemma}\label{lem:tangent:space}
Let $x$ and $y$ be in $Q$.  
Then, $y\in H_x$ if and only if $\;\widehat{y} \in \widehat{x}+H'_x\,$.  
\end{lemma}

\begin{lemma}\label{lem:coset:gap}
Let $x$ and $y$ be in $Q$.
Suppose that $\left\vert  x \right\vert\le\left\vert  y \right\vert,\, y\in H_x,$ 
    and $\|\widehat x-\widehat y\|_e\le\frac{4\lambda_1(x)}{\left\vert  x \right\vert}$.  
Then $\lambda_2(x)\asymp\lambda_2(y) $.  
\end{lemma}

\subsection{The first minimum of $\Lambda_y$}

In one dimension, when $\widehat{x}_n=\frac{p_{n}}{q_n}$ and $\widehat{x}_{n+1}=\frac{p_{n+1}}{q_{n+1}}$ are two consecutive convergents of a real number, the unimodular property of the ordinary continued fraction algorithm implies the two equivalent properties:
\begin{itemize}
\item[(i)] $\pi_{x_{n}}(p_{n+1},q_{n+1})=p_{n+1}-q_{n+1}\widehat x_n$ is one of the two primitive elements
of the lattice $\Lambda_{x_n}$,
\item[(ii)]  $\pi_{x_{n+1}}(p_{n},q_{n})=p_{n}-q_{n}\widehat x_{n+1}$ is a shortest vector of
$\Lambda_{x_{n+1}}$.
\end{itemize}

In higher dimensions, lattices have infinitly many primitive elements. So, a priori, given   two consecutive best approximation vectors $x$ and $y\in Q$
there are infinitely many possible primitive elements $\alpha\in \Lambda_x$ that could be the projection $\alpha=\pi_x(y)$. Moreover 
property (i) no longer imply property (ii). Lemma \ref{lem:sc:sv} below give an additional condition which, together with (i), implies (ii).

Given $x\in Q$ and a primitive element $\alpha$ in $\Lambda_x$, we let 
  $$
  \Lambda_{\alpha^\bot}=\pi^\perp_\alpha(\Lambda_x),
  $$ where 
  $\pi^\perp_\alpha$ is the orthogonal projection of $\R^2$ onto  
  the subspace $\alpha^\perp$ of vectors of $\R^2$ orthogonal to $\alpha$.  

For any $y\in Q$ such that $\pi_x(y)=\alpha$, the $1$-volume of 
  $\Lambda_{\alpha^\bot}$ satisfies 
  $$\vol(\Lambda_{\alpha^\bot})=\frac{\vol(\Lambda_x)}{\|\alpha\|_e}
      =\frac{1}{\|\alpha\|_e\left\vert  x \right\vert}=\frac{1}{|x\wedge y|}.$$  
Here, the quantity $|y\wedge z|$ is the $2$-volume of the orthogonal projection 
  of $y\wedge z\in\Lambda^2\R^{3}$ onto the subspace spanned by 
  $e_1\wedge e_{3}$ and $e_2\wedge e_{3}$.  
Equivalently, (see \S2 of \cite{cheung1})  
  $$|y\wedge z|=\left\vert  y \right\vert\left\vert  z \right\vert d(\widehat y,\widehat z).$$

Denote the first minimum of $\Lambda_{\alpha^\bot}$ 
  by $\lambda_1(\alpha)$. The following lemma was proved in \cite{cheche}. 
  
\begin{lemma}\label{lem:sc:sv} 
  Let $x\in Q$ and $\alpha$ be a primitive element of $\Lambda_x$.  
  Suppose that $y$ is an element in $Q$ such that $\pi_x(y)=\alpha$.  Then 
    $\frac{|x\wedge y|}{\left\vert  y \right\vert}\le\lambda_1(\alpha)$ implies 
    $\lambda_1(y)=\frac{|x\wedge y|}{\left\vert  y \right\vert}=\|\pi_y(x)\|_e$.  
\end{lemma}

\section{Upper bound for the Hausdorff dimension \label{sec:upper-bound}}

Let $\mu'>\mu>0$ be two real numbers. We want to define a self-similar covering $(J,\sigma,B)$ of the set $\SSing^{\ast}(\mu')$.  Since the sequence of best approximation vectors $(x_n)_{n\in \N}$ of any $\theta \in \R^2 $ converges to $\theta$, it is natural to choose a self similar structure such that all the sequences of best approximations vectors of the $\theta \in  \SSing^{\ast}(\mu')$ are admissible. Moreover, according to Corollary \ref{cor:sing}, all the best approximation vectors of $\theta \in  \SSing^{\ast}(\mu')$ are in the set
$$Q_{\mu}=\{x\in Q:\lambda_{1}(x)\leq\left\vert x\right\vert
^{-\mu}\},$$
hence $J=Q_\mu$ is a natural choice.  The maps $\sigma$ and $B$ are more difficult to defined. Using the extension to higher dimensions, of Legendre's Theorem (Lemma \ref{lem:BAI:2}) it is tempting to defined the map $B$ with $B(x)=B\left(\widehat x,\frac{2\lambda_1(x)}{\left\vert  x \right\vert}\right)$.
However, by a result of Jarn\'\i k \cite{Jar54},   if the uniform exponent $\widehat{\omega}_{1,2}(\theta)$ is $\geq \mu$, then the standard exponent of approximation $ \omega_{1,2}(\theta)$ is larger than
\[
 \frac{\mu^2}{1-\mu}.
\]
Therefore using subsequences of  sequences of best approximation vectors, it should be possible 
to define the sets $B(x)$ with smaller diameters.
The precise definition  involves  the subspaces $H_x$ defined section \ref{S:hyperplane}. \medskip

\textbf{Notation. }
\newline$E(x)=\{y\in Q_{\mu}:\left\vert y\right\vert >\left\vert
x\right\vert ,\ y\notin H_{x},\ \left\Vert \pi_{y}(x)\right\Vert_e \leq\frac
{1}{\left\vert y\right\vert ^{\mu}},\ \pi_x(y)$ is primitive in $\Lambda_x\}$,
\newline%
$D(y)=\{z\in Q_{\mu}:\left\vert z\right\vert \geq\left\vert y\right\vert ,\ z\in
H_{y},$\ $\left\Vert \widehat{y}-\widehat{z}\right\Vert_e \leq4\frac{\lambda
_{1}(y)}{\left\vert y\right\vert }\}$.

\begin{definition}
	\label{def:sigma-upper}
We set $\sigma_{\mu}(x)=\cup_{y\in E(x)}D(y)$ and $B_{\mu, c} (x)=B(\widehat{x},\frac
{c}{(\lambda_{2}(x)^{\mu}\left\vert x\right\vert )^{\frac{1}{1-\mu}}})$.
\end{definition}

\begin{remark}
In \cite{cheche}, the roles of $D$ and $E$ were permuted and $\sigma(x)$ was defined as
\[
\sigma_{\mu}(x)=\cup_{y\in D(x)}E(y).
\]
\end{remark} 

\begin{remark}
When $\lambda_1(x)\leq \left\vert x \right\vert^{-\mu}$, using the second Minkowski Theorem, it is easy to see that the radius of the ball $B_{\mu,c}(x)$ is $\ll$ 
\[ 
 \left\vert x \right\vert^{-(1+\frac{\mu^2}{1-\mu})}
\]
which is precisely what is expected from the result of Jarn\'\i k quoted above.
\end{remark}

Theorem \ref{thm:upper} is a consequence of the following two lemmata.

\begin{lemma}
	\label{lem:upper-self}
When $c$ is large enough, 
$(Q_{\mu},\sigma_{\mu},B_{\mu, c})$ is a self-similar covering of
$\SSing^{\ast}(\mu^{\prime})$ for all $\mu^{\prime}>\mu$.
\end{lemma}

\begin{proof}Let $\theta\in\SSing^{\ast}(\mu^{\prime})$ and
let $((p_{n},q_{n}))_{n \ge 0}$ be the sequence of best
approximations of
$\theta$. For $n \ge 0$, set $x_{n}=(p_{n},q_{n})$. 
By Corollary \ref{cor:sing} and removing the first best
approximation vectors if necessary, we
can suppose that $x_{n}\in Q_{\mu}$ for all $n$. Consider a subsequence
$(x_{n_{i}})_{i\geq0}$ such that for all $i\geq1$,
\[
x_{n_{i}+1}\notin H_{x_{n_{i}}},\ x_{n_{i}+1}, \ldots ,x_{n_{i+1}}\in
H_{x_{n_{i}+1}},\ x_{n_{i+1}+1}\notin H_{x_{n_{i}+1}}.
\]
Such a subsequence exists since the sequence $(x_{n})_{n \ge 0}$ must leave each subspace
$H_{x_{k}}$: otherwise the coordinates of the point $\theta$ together with $1$
would be rationally dependent. Observe that
\[
H_{x_{n_{i}+1}}=H_{x_{n_{i}+2}}= \ldots =H_{x_{n_{i+1}}}\neq H_{x_{n_{i+1}+1}}.
\]
Let $i$ be an integer. Set $x=(p,q)=x_{n_{i}}$, $y=(u,v)=x_{n_{i}+1}$ and
$z=x_{n_{i+1}}$. We have $y\notin H_{x}$ and, by Corollary \ref{cor:sing},
\[
\left\Vert q\widehat{y}-p\right\Vert_e \leq\frac
{1}{\left\vert y\right\vert ^{\mu}}. 
\]
Since $x$ and $y$ are consecutive best approximation vectors, $\pi_x(y)$ 
is primitive in $\Lambda_x$, hence $y\in E(x)$. Let $(e_{1},e_{2})$ 
be a reduced basis of $\Lambda_{x}$ and
$\alpha=\pi_{x}(y)$. Since $y\notin H_{x}$ we have $\alpha=ae_{1}+be_{2}$, 
where $b$ is a nonzero integer. We have
\[
\frac{\left\Vert \alpha\right\Vert_e \left\vert x\right\vert }{\left\vert
y\right\vert }=\frac{\left\vert x\wedge y\right\vert }{\left\vert y\right\vert
}=\left\Vert q\widehat{y}-p\right\Vert_e \leq \left\vert y\right\vert ^{-\mu},
\]
hence%
\[
\left\vert y\right\vert \geq(\left\Vert \alpha\right\Vert_e
\left\vert x\right\vert )^{\frac{1}{1-\mu}}%
\]
and
\[
\frac{\left\vert y\right\vert }{\left\vert x\right\vert }\geq(\left\Vert \alpha\right\Vert_e \left\vert x\right\vert ^{\mu})^{\frac
{1}{1-\mu}}.
\]
It follows that $y=\alpha+kx$, where the real number $k$ satisfies $\left\vert
k\right\vert \geq(\left\Vert \alpha\right\Vert_e \left\vert x\right\vert ^{\mu
})^{\frac{1}{1-\mu}}$. Moreover, %
\[
\widehat{y}=\widehat{x}+\frac{\alpha}{\left\vert y\right\vert }.
\]
Since $\left\Vert \alpha\right\Vert_e \gg\lambda_{2}(x)$, we get 
\begin{align*}
d(\widehat{x},\widehat{y}) &  \ll\frac{\left\Vert \alpha\right\Vert_e
}{(\left\Vert \alpha\right\Vert_e \left\vert x\right\vert )^{\frac{1}{1-\mu}}%
}=\frac{1}{(\left\Vert \alpha\right\Vert_e ^{\mu}\left\vert x\right\vert
)^{\frac{1}{1-\mu}}}\\
&  \ll\frac{1}{(\lambda_{2}(x)^{\mu}\left\vert x\right\vert )^{\frac{1}{1-\mu
}}}.
\end{align*}
Furthermore, $\theta\in B(\widehat{y},\frac{2\lambda_{1}(y)}{\left\vert y\right\vert
})$ and
\begin{align*}
\frac{\lambda_{1}(y)}{\left\vert y\right\vert } &  \ll\frac{1}{\left\vert
y\right\vert ^{1+\mu}}\leq\frac{1}{(\left\Vert \alpha\right\Vert_e \left\vert
x\right\vert )^{\frac{1+\mu}{1-\mu}}}\\
&  \ll\frac{1}{(\lambda_{2}(x)^{\mu}\left\vert x\right\vert )^{\frac{1}{1-\mu
}}}\times\frac{1}{(\lambda_{2}(x)\left\vert x\right\vert ^{\mu})^{\frac
{1}{1-\mu}}}. 
\end{align*}
Since $\mu\geq\frac{1}{2}$, we deduce from Minkowski's Theorem that
$$
\lambda_{2}(x)\left\vert
x\right\vert ^{\mu}\geq \lambda_{2}(x)\left\vert
x\right\vert ^{1-\mu} \gg 1,
$$
which implies that $\theta$ is in $B(\widehat{x},\frac
{c}{(\lambda_{2}(x)^{\mu}\left\vert x\right\vert )^{\frac{1}{1-\mu}}})$ when
$c$ is large enough. The last thing to check is that $z\in D(y)$, but this
follows from Lemma \ref{lem:BAI:3} (ii). 
\end{proof}

It appears that in some cases, it is better to use 
a larger radius for the balls $B_{c,\mu}$. This
observation has already been done in \cite{cheche}.
Since $\lambda_{2}(x)\gg\left\vert x\right\vert ^{\mu-1}$ for $x\in Q_{\mu}$,
a convex interpolation between the exponents of $\lambda_{2}(x)$ and $\left\vert x\right\vert ^{\mu-1}$ yields

\begin{corollary}
	\label{cor:upper-self}
For $\gamma\in [0,1]$ and $x\in Q_{\mu}$ set
\[
B_{\mu,\gamma}(x)=B(x)=B \Bigl( \widehat{x},\frac{c}{(\lambda_{2}(x)^{(1-\gamma)\mu}\left\vert
x\right\vert ^{(\mu-1)\mu \gamma+1})^{\frac{1}{1-\mu}}} \Bigr).
\]
When $c$ is large enough, 
$(Q_{\mu},\sigma_{\mu},B_{\mu,\gamma})$ is a self-similar covering of
$\SSing^{\ast}(\mu^{\prime})$ for all $\mu<\mu^{\prime}$.
\end{corollary}

\begin{lemma} 
\label{lem:sum-upper}
Let $a$\ and $b$ be real numbers with $b>2$ and
$\frac{b-1}{1-\mu}-a>2$. Then,  for $x\in Q_{\mu}$ with $\left\vert x\right\vert $
large enough, we get
\[
\sum_{z\in\sigma_{\mu}(x)}\frac{1}{\lambda_{2}(z)^{a}\left\vert z\right\vert ^{b}%
}\ll\frac{1}{\lambda_{2}(x)^{A}\left\vert x\right\vert ^{B}}, %
\]
where $A=\frac{b-1}{1-\mu}-a-2$ and $B=\mu\frac{b-1}{1-\mu}-a-1+b$.
\end{lemma}

\begin{proof}
\textbf{Step 1. }For $z\in D(y)$, we have $\lambda_{2}(z)\asymp
\lambda_{2}(y)$ because $z\in H_{y}$. It follows that
\begin{align*}
S_{1}(y)  &  =\sum_{z\in D(y)}\frac{1}{\lambda_{2}(z)^{a}\left\vert
z\right\vert ^{b}}\\
&  \asymp\sum_{z\in D(y)}\frac{1}{\lambda_{2}(y)^{a}\left\vert z\right\vert
^{b}}.
\end{align*}
For $z=(p,q) \in D(y)$, we have 
\[
\left\Vert \pi_{y}(z)\right\Vert_e=\left\Vert p-q\widehat{y}\right\Vert_e=
q\left\Vert \widehat{y}-\widehat{z}\right\Vert_e \leq 4q\frac{\lambda
_{1}(y)}{\left\vert y\right\vert },
\]
and since $\pi_y(z) \in \Lambda_y$, the number of elements in%
\[
D_{k}(y)=\{z\in D(y):k\left\vert y\right\vert \leq\left\vert z\right\vert
<(k+1)\left\vert y\right\vert \}
\]
is $\ll k$. It follows that
\begin{align*}
S_{1}(y)  &  \asymp\sum_{k\geq 1}\sum_{z\in D_{k}(y)}\frac{1}{\lambda
_{2}(y)^{a}\left\vert z\right\vert ^{b}}=\frac{1}{\lambda_{2}(y)^{a}\left\vert
y\right\vert ^{b}}\sum_{k\geq 1}\sum_{z\in D_{k}(y)} \Bigl( \frac{\left\vert
y\right\vert }{\left\vert z\right\vert } \Bigr)^{b}\\
&  \ll\frac{1}{\lambda_{2}(y)^{a}\left\vert y\right\vert ^{b}}\sum_{k\geq
1}\frac{1}{k^{b-1}}.
\end{align*}
Since $b>2$, we get 
\[
S_{1}(y)\ll\frac{1}{\lambda_{2}(y)^{a}\left\vert y\right\vert ^{b}}.
\]

\textbf{Step 2. }By the definition of $\sigma_{\mu}(x)$ and by step 1, we have%
\begin{align*}
S(x)  &  =\sum_{z\in\sigma_{\mu}(x)}\frac{1}{\lambda_{2}(z)^{a}\left\vert
z\right\vert ^{b}}=\sum_{y\in E(x)}\sum_{z\in D(y)}\frac{1}{\lambda_{2}%
(z)^{a}\left\vert z\right\vert ^{b}}\\
&  \ll\sum_{y\in E(x)}\frac{1}{\lambda_{2}(y)^{a}\left\vert y\right\vert ^{b}%
}\\
&  =\sum_{\substack{\alpha\in\Lambda_{x}\backslash H_{x}^{\prime}%
\\\alpha\text{ primitive}}}\sum_{y\in E(x)\ :\ \pi_{x}(y)=\alpha}\frac
{1}{\lambda_{2}(y)^{a}\left\vert y\right\vert ^{b}}.
\end{align*}
By the definition of $E(x)$, if $y\in E(x)$, then we have $\left\Vert \pi
_{y}(x)\right\Vert_e \leq\frac{2}{\left\vert y\right\vert ^{\mu}}$ and
$$
\left\Vert \alpha\right\Vert_e =\frac{\left\vert x\wedge y\right\vert
}{\left\vert x\right\vert }=\frac{\left\Vert p-q\widehat{y}\right\Vert_e
\left\vert y\right\vert }{\left\vert x\right\vert }=\frac{\left\Vert \pi
_{y}(x)\right\Vert_e \left\vert y\right\vert }{\left\vert x\right\vert }, 
$$
hence%
\[
\left\Vert \alpha\right\Vert_e \left\vert x\right\vert \leq2\left\vert
y\right\vert ^{1-\mu}%
\]
and
\[
\frac{\left\vert y\right\vert }{\left\vert x\right\vert }\geq(\frac{1}%
{2}\left\Vert \alpha\right\Vert_e \left\vert x\right\vert ^{\mu})^{\frac
{1}{1-\mu}}.
\]

Since $\frac{\left\vert x\right\vert \left\Vert \alpha\right\Vert_e
}{\left\vert y\right\vert }=\|\pi_y(x)\|_e\geq \lambda_1(y)$, 
we deduce from Minkowski's Theorem that 
\[
\lambda_{2}(y)\gg\frac{1}{\left\Vert \alpha\right\Vert_e \left\vert
x\right\vert }%
\]
holds for all $y\in E(x)$ such that $\pi_{x}(y)=\alpha$. Call $\lambda_{1}(\alpha)$
the first minimum of the orthogonal projection of $\Lambda_{x}$ on the line
orthogonal to $\alpha$. By Lemma \ref{lem:sc:sv}, if $\frac{\left\vert x\wedge
y\right\vert }{\left\vert y\right\vert }<\lambda_{1}(\alpha)$ then
$\lambda_{1}(y)=\frac{\left\vert x\wedge y\right\vert }{\left\vert
y\right\vert }=\frac{\left\vert x\right\vert \left\Vert \alpha\right\Vert_e
}{\left\vert y\right\vert }$, which implies that $\lambda_{2}(y)\asymp\frac
{1}{\left\Vert \alpha\right\Vert_e \left\vert x\right\vert }$. Now $\lambda
_{1}(\alpha)=\frac{1}{\left\Vert \alpha\right\Vert_e \left\vert x\right\vert }$
and $\mu>\frac{1}{2}$, hence, for $\left\vert x\right\vert $ large enough,
\begin{align*}
\left\vert y\right\vert  &  >(\frac{1}{2}\left\Vert \alpha\right\Vert_e
\left\vert x\right\vert )^{\frac{1}{1-\mu}}\Rightarrow\left\vert y\right\vert
>(\left\Vert \alpha\right\Vert_e \left\vert x\right\vert )^{2}\\
&  \Rightarrow\frac{\left\Vert \alpha\right\Vert_e \left\vert x\right\vert
}{\left\vert y\right\vert }<\frac{1}{\left\Vert \alpha\right\Vert_e \left\vert
x\right\vert }\\
&  \Rightarrow\frac{\left\vert x\wedge y\right\vert }{\left\vert y\right\vert
}<\lambda_{1}(\alpha).
\end{align*}

It follows that%
\begin{align*}
S  &  \ll\sum_{\alpha\in\Lambda_{x}\backslash H_{x}^{\prime}}\sum_{y\in
E(x):\pi_{x}(y)=\alpha}\frac{1}{(\frac{1}{\left\vert x\right\vert \left\Vert
\alpha\right\Vert_e })^{a}\left\vert y\right\vert ^{b}}\\
&  \asymp\sum_{\alpha\in\Lambda_{x}\backslash H_{x}^{\prime}}\sum_{k\geq
(\frac{1}{2}\left\Vert \alpha\right\Vert_e \left\vert x\right\vert ^{\mu
})^{\frac{1}{1-\mu}}}\frac{(\left\vert x\right\vert \left\Vert \alpha
\right\Vert_e )^{a}}{\left\vert x\right\vert ^{b}k^{b}}\\
&  \ll\sum_{\alpha\in\Lambda_{x}\backslash H_{x}^{\prime}}\frac{(\left\vert
x\right\vert \left\Vert \alpha\right\Vert_e )^{a}}{\left\vert x\right\vert
^{b}(\frac{1}{2}\left\Vert \alpha\right\Vert_e \left\vert x\right\vert ^{\mu
})^{\frac{b-1}{1-\mu}}}\\
&  \ll\sum_{\left\Vert \alpha\right\Vert_e \geq\lambda_{2}(x)}\frac
{1}{\left\vert x\right\vert ^{\mu\frac{b-1}{1-\mu}-a+b}\left\Vert
\alpha\right\Vert_e ^{\frac{b-1}{1-\mu}-a}}.
\end{align*}

Now $\frac{b-1}{1-\mu}-a>2$ if $s>\frac{3-2\mu}{1-\mu+\mu^{2}}(1-\mu)$.
Therefore, by Lemma 2.4 of \cite{cheche} about  sums over lattices,%
\begin{align*}
S  &  \ll\frac{1}{\left\vert x\right\vert ^{\mu\frac{b-1}{1-\mu}%
-a+b}\operatorname{vol}\Lambda_{x}\lambda_{2}(x)^{\frac{b-1}{1-\mu}-a-2}}\\
&  =\frac{1}{\lambda_{2}(x)^{A}\left\vert x\right\vert ^{B}}, %
\end{align*}
where $B=\mu\frac{b-1}{1-\mu}-a-1+b$ and $A=\frac{b-1}{1-\mu}-a-2$.
\end{proof}

\begin{proof}[Completion of proof of Theorem \ref{thm:upper}]

Let $\mu_0$ be in $(\frac12,1)$. \newline
\textbf{Case 1. }
Assume that $\mu_0>\frac{1}{\sqrt{2}}$. 
By Lemma \ref{lem:upper-self}, $(Q_{\mu},\sigma_{\mu},B_{\mu})$ 
is a self-similar covering of
$\Sing^{\ast}(\mu_0)$ for all  $\mu$ such that
$\frac{1}{\sqrt{2}}<\mu<\mu_0$.
Let $s>2(1-\mu)$.  Set $t=\frac
{s}{1-\mu}$, $a=\mu t$ and $b=t$. For $x\in Q_{\mu}$, set 
\[
S(x)=\sum_{z\in\sigma_{\mu}(x)}(\diam B(z))^{s}.
\]
With these notations, $(\diam B(x))^s=\frac{c^s}{\lambda_{2}%
(x)^{a}\left\vert x\right\vert ^{b}}$ for all $x\in\Q_{\mu}$, hence by  Lemma \ref{lem:sum-upper}, we have
\[
\frac{S(x)}{(\diam B(x))^{s}}\ll\frac{1}{\lambda_{2}%
(x)^{A-a}\left\vert x\right\vert ^{B-b}}.
\]

Straightforward calculations give
\begin{align*}
A-a  &  =\frac{b-1}{1-\mu}-2a-2\\
&  =\frac{1}{1-\mu}\left(  t(1-2\mu+2\mu^{2})+2\mu-3\right)  \allowbreak
\end{align*}
$\allowbreak$and
\begin{align*}
B-b  &  =\mu\frac{b-1}{1-\mu}-a-1\\
&  =\frac{t\mu^{2}-1}{1-\mu}. %
\end{align*}
By assumption $t>2$ and $\mu^{2}>\frac{1}{2}$, so $B-b$ is positive. If
$A-a<0$, then $S(x)\leq (\diam B(x))^{s}$ when $\vert x \vert$ is large enough. 
Otherwise we use that $\lambda_{2}(x)\gg\left\vert
x\right\vert ^{\mu-1}$ and we get
\[
\frac{1}{\lambda_{2}(x)^{A-a}\left\vert x\right\vert ^{B-b}}\ll\frac
{1}{\left\vert x\right\vert ^{C}}, %
\]
with
\begin{align*}
C  &  =(\mu-1)(A-a)+(B-b)\\
&  =\frac{2\mu-1}{1-\mu}\left(  t(1-\mu+\mu^{2})+\mu-2\right) \\
&  >\frac{2\mu-1}{1-\mu}(-\mu+2\mu^{2})>0.
\end{align*}
We conclude that  $S(x)\leq (\diam B(x))^{s}$ 
when $\vert x \vert$ is large enough. Therefore, by Theorem \ref{thm:SSCupper}, 
\[
\dim_H\Sing^{\ast}(\mu_0)\leq s
\]
and since this holds for all $s>2(1-\mu)$ and all $\frac{1}{\sqrt{2}}<\mu<\mu_0$, we obtain
\[
\dim_H\Sing^{\ast}(\mu_0)\leq 2(1-\mu_0).
\]
\textbf{Case 2. } Assume that $\mu_0<\frac{1}{\sqrt{2}}$.
We use Corollary \ref{cor:upper-self} instead of Lemma \ref{lem:upper-self} 
with $\mu<\mu_0$ and a suitable choice of $\gamma$.
Set $t=\frac{s}{1-\mu}$, $a=(1-\gamma)\mu t$ and $b=(1+\gamma(\mu-1)\mu)t$. The idea is to find
a value of $\gamma$ such that the constraints
\[
b>2,\ \frac{b-1}{1-\mu}-a>2,\ B-b>0
\]
are satisfy with $t$ minimal. This leads to the value $\gamma=\frac{1-2\mu^{2}}%
{\mu(1-\mu)(3-2\mu)}$. In fact with the value $t=\frac{3-2\mu}{1-\mu+\mu^{2}}$
we find $b=2$, $\frac{b-1}{1-\mu}-a=2$, and $B-b=0$. It follows that if
$t>\frac{3-2\mu}{1-\mu+\mu^{2}}$ the three strict inequalities hold. The last
thing to check is that with this value of $\gamma$ and $t>\frac{3-2\mu}{1-\mu
+\mu^{2}}$ we have $A-a\leq0$. Now, if $t=\frac{3-2\mu}{1-\mu+\mu^{2}}$ we have
$A-a=\frac{1}{\mu-1}\left(  2\mu-1\right)  <0$, hence $A-a<0$ for $t$ close to
$\frac{3-2\mu}{1-\mu+\mu^{2}}$ which implies that $S(x)\leq1$ for $\left\vert
x\right\vert $ large enough. 
\end{proof}

\section{Lower bounds for the Hausdorff dimension: tools}

\subsection{The counting/diameter function}

We will use Theorem \ref{thm:massdistribution} when all the diameters of the sets 
$B(z)$, $z\in\sigma(x)$, 
have the same order. In that case we can replace the sums $\sum_{z\in\sigma
_{F}(x)}(\diam B(z))^{s}$  in condition (iii) of Theorem \ref{thm:massdistribution} 
by an equivalent sum
\[
(\operatorname{diam}B(z))^{s}\times\operatorname{card}\{z\in\sigma(x):B(z)\cap
F\neq\emptyset\}.
\]
So we are reduced to bound 
$\operatorname{card}\{z\in\sigma(x):B(z)\cap
F\neq\emptyset\}$ from above with $(\operatorname{diam}F)^{s}$. This will be done when
the  $z\in\sigma(x)$ are on line segments through some points in almost lattice positions. 
The next lemma allows us to bound from above $\sum_{z\in\sigma_{F}(x)}%
\frac{(\diam B(z))^{s}}{(\diam F)^{s}}$ in such a situation.

\begin{definition}
Let $C_{0}\geq 1,\ H>0$ and $V>0$ be real numbers. A $C_{0}$-distorted $H\times
V$-tiling of a subset $\mathcal{B}$ in $\R^{2}$ is a finite collection
of subsets $\mathcal{R}_{i}$, $i\in I$, such that
\begin{itemize}
\item[1.] each
$\mathcal{R}_{i}$ is included in $\mathcal{B}$,
\item[2.] the intersection of $\mathcal{R}_{i}$
and $\mathcal{R}_{j}$ has measure zero for all $i\neq j$,
\item[3.] each $\mathcal{R}_{i}$ contains a rectangle of horizontal length
$\frac{1}{C_{0}}H$ and of vertical length $\frac{1}{C_{0}}V$,
\item[4.] each
$\mathcal{R}_{i}$ is contained in a rectangle of horizontal length $C_{0}H$
and of vertical length $C_{0}V$.
\end{itemize}
\end{definition}

\textbf{Assumptions of Lemma \ref{lem:counting}.} 
Let $C_{0}\geq1$ be a real number, let $R_{0}>R_{1}>R_{2}>R_{3}$ and $H,V$ be
real numbers such that%
\[
\frac{R_{0}}{C_{0}}\geq H,\quad V\geq\frac{R_{1}}{C_{0}}, 
\]
and let $\mathcal{E}$ be a finite subset of $\R^{2}$. Assume that
$(\mathcal{R}_{y})_{y\in\mathcal{E}}$ is a $C_{0}$-distorted $H\times V$
tiling of the ball $B(x,R_{0})$ such that each set $\mathcal{R}_{y}$ contains
the corresponding $y$ of $\mathcal{E}$. Furthermore assume that, 
for each $y\in \mathcal{E}$, the ball
$B(y,R_{1})$ contains a set of balls $B(z_{1},R_{3}), \ldots , B(z_{k_y},R_{3})$,
$k_y\leq\lfloor\frac{2R_{1}}{R_{2}}\rfloor$, which are disjoint and whose
centers $z_{i}$ are in a same line going through $y$, the distance between
consecutive centers being at least $R_{2}$. Call $\mathcal{D}_{y}$ the set of all
the $z_{i}$ and set 
\[
\mathcal{S=\cup}_{y\in\mathcal{E}}\mathcal{D}_{y}.
\]

\begin{lemma}
\label{lem:counting}Set $f(r)=\max_{a\in\R^{2}}\frac
{\card \mathcal{S}\cap B(a,r)}{r^{s}}$. 
\begin{itemize}
\item[1.] If $1 \leq s\leq 2$, then
\[
\max_{R_{3}\leq r\leq R_{0}}f(r)\leq 72C_{0}^{4}\max \Bigl\{\frac{1}{R_{3}^{s}}%
,\frac{R_{1}R_{0}^{2}}{VHR_{2}}\times\frac{1}{R_{0}^{s}} \Bigr\}.
\]
\item[2.] If $s<1$, then
\[
\max_{R_{3}\leq r\leq R_{0}}f(r)\leq 72C_{0}^{4}\max \Bigl\{\frac{1}{R_{3}^{s}}%
,\frac{R_{1}}{R_{2}R_{1}^{s}},\frac{R_{1}R_{0}^{2}}{VHR_{2}}\times\frac
{1}{R_{0}^{s}} \Bigr\}.
\]
\end{itemize}

\end{lemma}

\begin{proof} We can assume that $V\leq H$.

Observe first that a $4C_{0}H\times 4C_{0}V$ rectangle can meet at most
$36C_{0}^{4}$ tiles $\mathcal{R}_{y}$ because the union of all these tiles is 
included in a $6C_{0}H\times 6C_{0}V$ rectangle and these tiles have an area at
least equal to $C_{0}^{-2}HV$. Next, if a ball $B(a,r)$ meets a ball $B(y,R_{1})$ with
$y\in\mathcal{E}$, then the ball $B(a,r+R_{1})$ meets the tile $\mathcal{R}%
_{y}$. Since a ball $B(a,r+R_{1})$ with $r\leq C_{0}V$ is included in a
$4C_{0}H\times 4C_{0}V$ rectangle, it follows that a ball $B(a,r)$ with $r\leq
R_{1}$ meets at most
$36C_{0}^{4}$
balls $B(y,R_{1})$, $y\in\mathcal{E}$.

\textbf{Case 1. }$R_{3}\leq r\leq R_{2}.$\newline Since, for a given $y$ in
$\mathcal{E}$, a ball $B(a,r)$ contains at most two points $z$ in
$\mathcal{D}_{y}$, by the above observation we have%
\[
f(r)\leq 72C_{0}^{4}\times r^{-s}=g(r), 
\]
which is a decreasing function of $r$.

\textbf{Case 2. }$R_{2}\leq r\leq R_{1}$.\newline Since, for a given $y$ in
$\mathcal{E}$, a ball $B(a,r)$ contains at most $\frac{2r}{R_{2}}$ points $z$
in $\mathcal{D}_{y}$, by the above observation we have%
\[
f(r)\leq\frac{36C_{0}^{4}}{r^{s}}\times\frac{2\times r}{R_{2}}=72C_{0}%
^{4}\times\frac{r^{1-s}}{R_{2}}=g(r), 
\]
\newline which is an increasing function of $r$ if $s\leq 1$, and a decreasing 
function otherwise.

\textbf{Case 3. }$R_{1}\leq r\leq C_{0}V$.\newline By the above observation we
have%
\[
f(r)\leq 36C_{0}^{4}\times\frac{2R_{1}}{R_{2}}\times r^{-s}=g(r), 
\]
\newline which is a decreasing function of $r$. 

\textbf{Case 4. }$C_{0}V\leq r\leq C_{0}H$.\newline We need first to refine
the above observation. A $2(r+R_{1})\times 2(r+R_{1})$ square is included in a
$4C_{0}H \times \frac{4r}{C_{0}V}C_{0}V$ rectangle and all the tiles meeting
this rectangle are included in a $(\frac{4r}{C_{0}V}+2)C_{0}V\times 6C_{0}H$
rectangle. It follows that the $2(r+R_{1})\times 2(r+R_{1})$ square meets at
most
\[
(\frac{6r}{C_{0}V}\times 6)C_{0}^{2}\frac{VH}{C_{0}^{-2}VH}=36C_{0}^{4}%
\times\frac{r}{C_{0}V}%
\]
tiles $\mathcal{R}_{y}$. Hence
\[
f(r)\leq 36C_{0}^{4}\times\frac{r}{C_{0}V}\times\frac{2R_{1}}{R_{2}}\times
r^{-s}= 72C_{0}^{3}\frac{R_{1}}{VR_{2}}r^{1-s}=g(r), 
\]
which is an increasing function of $r$ if $s\leq 1$, and a decreasing 
function otherwise.

\textbf{Case 5. }$C_{0}H\leq r\leq R_{0}$. \newline The number of tiles meets
by $2(r+R_{1})\times 2(r+R_{1})$ square is at most $\frac{36C_{0}^{2}r^{2}}%
{HV}$, hence%
\[
f(r)\leq\frac{\frac{36C_{0}^{2}r^{2}}{HV}\times\frac{2R_{1}}{R_{2}}}{r^{s}%
}= 72C_{0}^{2}\times\frac{R_{1}}{VHR_{2}}r^{2-s}=g(r)
\]
\newline which is a decreasing function of $r$. 

\textbf{Conclusion. }$\newline$If $s\geq 1$, then  $f(r)\leq g(r)\leq\max
(g(R_{3}),g(R_{0}))\leq 72C_{0}^{4}\max\{\frac{1}{R_{3}^{s}},\frac{R_{1}%
R_{0}^{2}}{VHR_{2}}\times\frac{1}{R_{0}^{s}}\}$. 
\newline If $s\leq 1$, the
maximum of $g$ might be reach in $r=R_{1}$. 
\end{proof}

The above lemma will be used with an $s$ chosen so that $\operatorname{card}%
\mathcal{S\times}R_{3}^{s}\geq R_{0}^{s}$. Thanks to Theorem
\ref{thm:massdistribution}, it gives a lower bound for the Hausdorff dimension
of the image of $\Omega_{x_{0}}$ when $s\geq 1$. If $s\leq 1$, it will be
necessary to check that
\[
1\gg\frac{g(R_{0})}{g(R_{1})}.\
\]

\subsection{A first step in the definition of the self-similar structure: 
definition of $\sigma$ and of $Q_{\sigma}$}

Let $\mu>\frac{1}{2}$ be fixed. We want to define 
a self-similar structure $(J,\sigma,B)$ that covers a subset of $\SSing^{\ast}(\mu)$. 
In this subsection we only define $J$ and $\sigma$.

We denote by $c_{1},c_{2}, \ldots$ some constants that will be chosen later. 
These constants might depend on $\mu$.
The constants involved in $\ll$,\ $\gg$, or in $\asymp$ depends only on  $\mu$ 
but not on $c_{1},c_{2}, \ldots$ \medskip

For each $x$ in $Q$ let $u_{1}=u_{1}(x),u_{2}=u_{2}(x)$ 
be a reduced basis of $\Lambda_{x}$
(by reduction, we mean the Gauss reduction). The vector subspace $H_{x}^{\prime}$ 
is spanned by
$u_{1}\mathbb{\ }$and $x$ (see the definition of $H_x$ in Section 3). 
\medskip

Let $E_{1}(x)$ be the set of $y=\alpha+kx\in Q$
with $\alpha=\alpha_{m}=\pi_{x}(y)=mu_{1}+u_{2}$ in the ``first level"
$u_{2}+H_{x}^{\prime}$, $\left\Vert mu_{1}\right\Vert_e \leq\lambda_{2}(x)$ and
$\left\vert y\right\vert \in(\left\Vert \alpha\right\Vert_e \left\vert
x\right\vert )^{\frac{1}{1-\mu}}[c_{0},2c_{0}]$ where $c_{0}=32^{\frac{1}{1-\mu}}$. 
This value of $c_0$ will be used in the proof of Proposition \ref{prop:SSS}.
\medskip

Fix $b$ a positive real number. Let $y$ be in $E_{1}(x)$. Let $D_{1}(y)$
be the set of $z$ in $Q$ such that $\left\vert z\right\vert \geq\left\vert
y\right\vert ,\ z\in H_{y}$, 
\begin{align*}
\frac{1}{2}\left\vert y\right\vert ^{b}\leq\frac{\left\vert z\right\vert
}{\left\vert y\right\vert }\leq\left\vert y\right\vert ^{b},
\\
\left\Vert \widehat{y}-\widehat{z}\right\Vert_e
\leq c_1 \frac{\lambda_{1}(y)}{\left\vert y\right\vert }, 
\end{align*}
where $c_1\leq \frac14$ is small enough 
and will be chosen after Lemma \ref{lem:distortedtiling}. 

For each $x$ in $Q$, set%
\[
\sigma(x)=\cup_{y\in E_{1}(x)}D_{1}(y)
\]
and
\[
J=Q_{\sigma}=\bigcup_{x\in Q,\,\left\vert  x \right\vert\geq c}\sigma(x), 
\]
where $c$ is a constant.
The remaining Propositions and Lemmas hold when $\left\vert  x \right\vert$ is
large enough, so the constant $c$ will be chosen in order all these results
hold.

\subsection{A few Calculations}

Let $x$ be in $Q_{\sigma}$, $y$ be in $E_{1}(x)$ and $z$ be in $D_{1}(y)$.
Since $\pi_{x}(y)=\alpha=mu_{1}+u_{2}$ with $\left\Vert mu_{1}\right\Vert_e
\leq\lambda_{2}(x)$, we have $\left\Vert \alpha\right\Vert_e \asymp\lambda
_{2}(x)$. So if we can evaluate $\lambda_{2}(x)$, we 
will then be able to estimate the height of $y$, the height of $z$ 
and also $\lambda_1(x)$. However, estimating $\lambda_2(x)$ 
is not possible directly and we have to estimate $\lambda_1(y)$ first.   

\subsubsection{Minima of $\Lambda_y$}

\begin{lemma}\label{lem:y}
Let $x$ be in $Q_{\sigma}$ and $y$ in $E_1(x)$. 
Then $\lambda_1(y)=\frac{|x\wedge y|}{\left\vert  x \right\vert}\asymp \left\vert  y \right\vert^{-\mu}$ 
and $\lambda_2(y)\asymp \left\vert  y \right\vert^{\mu-1}$ 
when $\left\vert  x \right\vert$ is large enough.
\end{lemma}

\begin{proof}
Let $\lambda_{1}(\alpha)$ denote 
the first minimum of the orthogonal projection of $\Lambda_{x}$ on the line
orthogonal to $\alpha$. 
By Lemma \ref{lem:sc:sv}, if $\frac{\left\vert x\wedge y\right\vert }{\left\vert
y\right\vert }\leq\lambda_{1}(\alpha)$,   
then $\lambda_{1}(y)=\frac{\left\vert
x\wedge y\right\vert }{\left\vert y\right\vert }=\left\Vert \alpha\right\Vert_e
\frac{\left\vert x\right\vert }{\left\vert y\right\vert }$. Now $\lambda
_{1}(\alpha)=\frac{\det\Lambda_{x}}{\left\Vert \alpha\right\Vert_e }=\frac
{1}{\left\Vert \alpha\right\Vert_e \left\vert x\right\vert }\geq\frac{\left\vert
x\wedge y\right\vert }{\left\vert y\right\vert }$ is equivalent to%
\[
\left\vert y\right\vert \geq(\left\Vert \alpha\right\Vert_e \left\vert
x\right\vert )^{2}%
\]
and, by definition of $E_{1}(x)$, we get $\left\vert y\right\vert \geq c_{0}%
(\left\Vert \alpha\right\Vert_e \left\vert x\right\vert )^{\frac{1}{1-\mu}}%
\geq(\left\Vert \alpha\right\Vert_e \left\vert x\right\vert )^{2}$ (note that $\left\Vert
\alpha\right\Vert_e \left\vert x\right\vert >1$ when $\left\vert x\right\vert
$ is large enough), therefore $\lambda_{1}(y)=\frac{\left\vert x\wedge
y\right\vert }{\left\vert y\right\vert }$ when $\left\vert x\right\vert $ is
large enough.

It follows that
\[
\lambda_{1}(y)  =\left\Vert \pi_{y}(x)\right\Vert_e =\left\Vert \alpha
\right\Vert_e \frac{\left\vert x\right\vert }{\left\vert y\right\vert }%
\asymp(\left\vert x\right\vert \left\Vert \alpha\right\Vert_e )^{1-\frac
{1}{1-\mu}}\asymp\left\vert y\right\vert ^{-\mu}
\]
and, by Minkowski's Theorem, 
\[
\lambda_{2}(y)   \asymp\frac{1}{\left\vert y\right\vert \lambda_{1}(y)}%
\asymp\left\vert y\right\vert ^{\mu-1}.
\]
\end{proof}

\subsubsection{Minima of $\Lambda_x$ and $\Lambda_z$}

\begin{lemma}\label{lem:xz}
Let $x$ be in $Q_{\sigma}$ and $z$ in $\sigma(x)$. 
We have $\lambda_1(z)\asymp\left\vert  z \right\vert^{-\frac{\mu+b}{1+b}}$ and 
$\lambda_2(z)\asymp\left\vert  z \right\vert^{\frac{\mu-1}{1+b}}$ 
when $\left\vert  x \right\vert$ is large enough. Consequently, 
$\lambda_1(x)\asymp\left\vert  x \right\vert^{-\frac{\mu+b}{1+b}}$ 
and $\lambda_2(x)\asymp\left\vert  x \right\vert^{\frac{\mu-1}{1+b}}$.
\end{lemma}

\begin{proof}
Let $y$ be in $E_1(x)$ such that $z$ is in $D_1(y)$. 
By definition of $D_1(y)$, we get $\left\vert z\right\vert \asymp\left\vert y\right\vert ^{1+b}$.
By Lemma \ref{lem:coset:gap} 
and by the definition of $D_1(y)$, and then by Lemma \ref{lem:y}, we have
\begin{align*}
\lambda_{2}(z)  &  \asymp\lambda_{2}(y)\asymp\left\vert y\right\vert ^{\mu
-1}\asymp\left\vert z\right\vert ^{\frac{\mu-1}{1+b}}, \\
\lambda_{1}(z)  &  \asymp\left\vert z\right\vert ^{-1-\frac{\mu-1}{1+b}%
}=\left\vert z\right\vert ^{-\frac{\mu+b}{1+b}}.
\end{align*}
\end{proof}

\subsubsection{Distance from $\widehat{x}$ to $\widehat{y}$}

\begin{lemma}\label{lem:d(x,y)}
Let $x$ be in $Q_{\sigma}$ and $y$ in $E_1(x)$. Then,
when $\left\vert  x \right\vert$ is large enough, we have 
\[
d(\widehat{x},\widehat{y})\asymp\left\vert  x \right\vert^{r_0}, 
\]
where
\[
r_{0}=-\frac{  \mu^{2}%
-\mu+b+1}{\left(  1-\mu\right)  \left(  b+1\right)  }  .
\]
\end{lemma}

\begin{proof}
Let $\alpha=\pi_x(y)$. Since $y=\alpha+\frac{\left\vert y\right\vert }{\left\vert x\right\vert }x$, we get
\[
d(\widehat{x},\widehat{y})=\frac{\left\Vert \alpha\right\Vert_e }{\left\vert
y\right\vert }\asymp\frac{\left\Vert \alpha\right\Vert_e }{(\left\vert
x\right\vert \left\Vert \alpha\right\Vert_e )^{\frac{1}{1-\mu}}}=\frac
{1}{(\left\vert x\right\vert \left\Vert \alpha\right\Vert_e ^{\mu})^{\frac
{1}{1-\mu}}}\asymp\frac{1}{(\left\vert x\right\vert \lambda_{2}^{\mu
}(x))^{\frac{1}{1-\mu}}}.
\]
Therefore, by Lemma \ref{lem:xz}, we have
\begin{align*}
d(\widehat{x},\widehat{y})  &  \asymp\frac{1}{(\left\vert
x\right\vert \left\vert  x \right\vert^{\frac{\mu-1}{1+b}\mu})^{\frac{1}{1-\mu}}}\\
&  \asymp\frac{1}{\left\vert x\right\vert ^{\frac{1}{1-\mu}(1+\frac{\mu
(\mu-1)}{1+b})}}=\left\vert x\right\vert ^{r_{0}}. %
\end{align*}
\end{proof}

\subsubsection{Growth rate of $\left\vert  y \right\vert$}

\begin{lemma}\label{lem:height}
Let $x$ be in $Q_{\sigma}$ and $y$ in $E_1(x)$. Then, 
when $\left\vert  x \right\vert$ is large enough, we get 
\[
\left\vert  y \right\vert=\left\vert  x \right\vert^{e_y},
\]
where
\[
e_y=\frac{\mu
+b}{(1-\mu)(1+b)}.
\]
\end{lemma}

\begin{proof}
By the definition of $E_1(x)$ and by Lemma \ref{lem:xz}, we have 
\[
\left |y \right|\asymp 
(\left\vert x\right\vert \lambda_{2}(x))^{\frac{1}{1-\mu}} 
\asymp (\left| x\right\vert \left\vert  x \right\vert^{\frac{\mu-1}{1+b}})^{\frac{1}{1-\mu}}.
\]
\end{proof}

\subsection{A nested self-similar structure}
We want to define a self-similar structure $(J,\sigma,B)$.
Since $J=Q_{\sigma}$ and $\sigma$ have already been defined, it remains 
only for us to define the map $B$.

\subsubsection{Definition of $B(x)$}

For each $x\in Q_{\sigma}$, set
\[
B(x)= B(\widehat{x},c_{2}\left\vert x\right\vert ^{r_{0}}), 
\]
where the constant $r_0$ is defined in Lemma \ref{lem:d(x,y)}. 
The constant $c_{2}$ will be chosen in the proof of Lemma \ref{lem:distortedtiling}.

\begin{lemma}\label{lem:B(x):best}
	For $x$ in $Q_{\sigma}$,%
	\[
	B(x)\subset B\Bigl( \widehat{x},\frac{\lambda_{1}(x)}{2\left\vert x\right\vert } \Bigr), 
	\]
	when $\left\vert x\right\vert $ large enough, and therefore $x$ is a best
	approximation vector of all $\theta$ in $B(x)$.
\end{lemma}

\begin{proof} By Lemma \ref{lem:xz}, 
$\lambda_1(x)\asymp \left\vert  x \right\vert^{-\frac{\mu+b}{1+b}}$ and 
\[
	-\frac{\mu+b}{1+b}-1+r_0=-\frac{\mu+2b+1}{1+b}+\frac{  \mu^{2}%
		-\mu+b+1}{\left(  1-\mu\right)  \left(  b+1\right)  }>0,
\]
Therefore $c_{2}\left\vert x\right\vert ^{r_{0}}\leq\frac{\lambda_{1}
(x)}{2\left\vert x\right\vert }$ when $\left\vert x\right\vert $ is large enough. 
By Lemma \ref{lem:BAI:2}, it follows that $x$ 
is a best approximation vector of all $\theta$ in $B(x)$.
\end{proof}

\subsubsection{Nestedness}

\begin{lemma}
\label{lem:nestedness}Let $x$ be in $Q_{\sigma}$. Then for all $z\in
\sigma(x)$, $B(z)\subset B(x)$ and $z\in Q_{\sigma}$ when $\left\vert
x\right\vert $ is large enough. Moreover, 
$$
B(z)\subset
B\Bigl( \widehat{y},\frac{\lambda_{1}(y)}{2\left\vert y\right\vert } \Bigr)
\subset B\Bigl(\widehat{y},\frac{2\lambda_{1}(y)}{\left\vert y\right\vert } \Bigr) \subset B(x), 
$$
where $y$ is the element of $E_{1}(x)$ such that $z\in D_{1}(y)$.
\end{lemma}

\begin{proof}
By Lemma \ref{lem:d(x,y)},
\[
d(\widehat{x},\widehat{y})\ll\left\vert x\right\vert ^{r_{0}}, 
\]
where $r_{0}$ $=-\frac{\mu^{2}-\mu+b+1}{\left(  1-\mu\right)  \left(  b+1\right)  }$. 
Moreover, by Lemmas \ref{lem:y} and \ref{lem:height}, $\frac{\lambda_{1}(y)}{\left\vert
y\right\vert }\asymp\left\vert x\right\vert ^{-\frac{(\mu+b)(1+\mu)}%
{(1-\mu)(1+b)}}$ and
\[
r_{0}+\frac{(\mu+b)(1+\mu)}{(1-\mu)(1+b)}=\frac{2\mu-1+b\mu}{(1-\mu)(1+b)}>0.
\]
Therefore%
\[
B(\widehat{y},\frac{2\lambda_{1}(y)}{\left\vert y\right\vert })\subset B(x), 
\]
for all $y\in E_{1}(x)$ when $\left\vert x\right\vert $ is large enough.

Now let $y$ be in $E_{1}(x)$ and $z$ be in $D_{1}(y)$. By definition we have
$d(\widehat{y},\widehat{z})\leq \frac14 \frac{\lambda_{1}(y)}{\left\vert
y\right\vert }$, so in order to prove that $B(z)\subset B(\widehat{y}%
,\frac{\lambda_{1}(y)}{2\left\vert y\right\vert })\subset B(x)$ it is enough
to prove that $c_{2}\left\vert z\right\vert ^{r_{0}}\leq \frac14 %
\frac{\lambda_{1}(y)}{\left\vert y\right\vert }$.

Since $\lambda_{1}(y)\asymp\left\vert y\right\vert ^{-\mu}$
(by Lemma \ref{lem:y}), we are reduced to
check that%
\[
c_{2}\left\vert y\right\vert ^{(1+b)r_{0}}\asymp c_{2}\left\vert z\right\vert
^{r_{0}}\leq\frac{1}{4\left\vert y\right\vert ^{1+\mu}}, 
\]
which holds when $\left\vert y\right\vert $ is large enough because
\[
-(1+b)r_{0}-(1+\mu)=\frac{1}{1-\mu}\left(  (2\mu-1)\mu+b\right)  >0.
\]
\end{proof}

\begin{proposition}
\label{prop:SSS}The self-similar structure $(Q_{\sigma},\sigma,B)$ is strictly
nested and covers a subset of $\SSing^{\ast}(\mu)$.
\end{proposition}

\begin{proof} The nestedness is ensured by the previous Lemma and the fact
that $\lim_{n\rightarrow\infty}\diam B(x_{n})=0$ for all
admissible sequence $(x_{n})_{n\ge 0}$ is an immediate consequence of
the inequality $\left\vert z\right\vert >\left\vert x\right\vert $ for all
$z\in\sigma(x)$.

Let $(x_{n})_{n\ge 0}$ be an admissible sequence and let $\theta$ be
the unique point in $\cap_{n\geq0}B(x_{n})$. We have to show that $\theta
\in\SSing(\mu)$ and that $\Z\theta+\Z^{2}$ is
everywhere dense in $\R^{2}$. Let $Q$ be an integer. We want to prove
that there exists an integer $q\leq Q$ such that%
\[
d(q\theta,\Z^{2})\leq Q^{-\mu}.
\]
Let $n$ be the integer defined by $\left\vert x_{n}\right\vert \leq Q<\left\vert
x_{n+1}\right\vert $.\ Set $x=x_{n}$. By the definition of $\sigma$, there exists
$y\in E_{1}(x)$ such that $z=x_{n+1}\in D_{1}(y)$.

\textbf{Case 1:} $\left\vert x_{n}\right\vert \leq Q<\left\vert y\right\vert
$. \newline By the above lemma $\theta \in B(z)\subset B(\widehat{y},\frac{\lambda
_{1}(y)}{2\left\vert y\right\vert })$, hence $y$ is a best approximation
vector of $\theta$. By Lemma \ref{lem:BAI:3} (iii), for all $(p,q)$ in
$\Z^{2}$ with $0<q<\left\vert y\right\vert $, we have 
\[
\left\Vert p-q\theta\right\Vert_e \leq2\left\Vert p-q\widehat{y}\right\Vert_e.
\]
Now by the definition of $E_{1}(x_{n})$ we have $\left\vert y\right\vert \geq
c_{0}(\left\Vert \alpha\right\Vert_e \left\vert x_{n}\right\vert )^{\frac
{1}{1-\mu}}$ where $\alpha=\pi_{x}(y)$. This implies that
\[
\left\Vert \widehat{y}-\widehat{x}\right\Vert_e \leq\frac{\left\vert
y\right\vert ^{-\mu}}{c_{0}^{1-\mu}\left\vert x\right\vert }%
\]
and therefore, by Lemma \ref{lem:BAI:3} (iii), the constant $c_{0}$ can be chosen large
enough so that 
\[
\left\Vert p-q\theta\right\Vert_e \leq2\left\Vert p-q\widehat{y}\right\Vert_e
\leq\left\vert y\right\vert ^{-\mu}\leq Q^{-\mu}.
\]
where $x=(p,q)$.

\textbf{Case 2: }$\left\vert y\right\vert \leq Q<\left\vert z\right\vert
$.\newline Since $d(\widehat{z},\widehat{y})\leq \frac14 \frac{\lambda_{1}%
(y)}{\left\vert y\right\vert }\leq\frac{\lambda_{1}(y)}{2\left\vert
y\right\vert }$, by Lemma \ref{lem:BAI:2}, $y$ is a best approximation vector of $\widehat{z}$. Let
$y_{0}=y=(p_{0},q_{0}),\ y_{1}=(p_{1},q_{1}), \ldots ,y_{k}=(p_{k},q_{k})=z$ be all
the intermediate best approximation vectors  of $\widehat{z}$. Since $\Lambda_{z}=\pi
_{z}(\Z^{3})\subset H_{y}^{\prime}+\Lambda_{y}^{\prime}$ and since by Lemma \ref{lem:y}, $
\lambda_1(y)$ is small compared to $\lambda_2(y)$ when $\left\vert  x \right\vert$
is large enough, the intermediate best approximation vectors are all in
$H_{y}$ and $\Lambda_{y_{i}}\subset\mathcal{H=}H_{y}+\Lambda_{y}$ ,
$i=0, \ldots ,k$. It follows that for each $i<k$ we have%
\[
\lambda_{1}(y_{i})e(\mathcal{H)=}\frac{1}{\left\vert y_{i}\right\vert }, %
\]
where $e(\mathcal{H)}$ is the distance between two consecutive lines of
$\mathcal{H}$. Since $e(\mathcal{H)\geq}\frac{\sqrt{3}}{2}\lambda_{2}(y)$, and
$\lambda_{1}(y)\lambda_{2}(y)\geq\frac{1}{\left\vert y\right\vert }$ (the
minima are associated with an Euclidean norm),
\[
\lambda_{1}(y_{i})=\frac{1}{e(\mathcal{H})\left\vert y_{i}\right\vert }%
\leq\frac{2}{\sqrt{3}}\frac{\lambda_{1}(y)\left\vert y\right\vert }{\left\vert
y_{i}\right\vert }\leq 2\frac{1}{c_{0}^{1-\mu}}\left\vert y\right\vert ^{-\mu
}\frac{\left\vert y\right\vert }{\left\vert y_{i}\right\vert }\leq 2\frac
{1}{c_{0}^{1-\mu}}\left\vert y_{i}\right\vert ^{-\mu}.
\]
Hence, by Lemma \ref{lem:BAI:3} (i) and (iii), for $\left\vert y_{i-1}\right\vert
\leq Q<\left\vert y_{i}\right\vert $, $i=1, \ldots ,k$,%
\begin{align*}
d(\{\widehat{z}, \ldots ,Q\widehat{z}\},\Z^{2})  &  \leq 2d(\{\widehat
{y_{i}}, \ldots ,Q\widehat{y}_{i}\},\Z^{2})\\
&  \leq 2\left\Vert q_{i-1}\widehat{y_{i}}-p_{i-1}\right\Vert_e \leq 8\lambda
_{1}(y_{i})\\
&  \leq 16\frac{1}{c_{0}^{1-\mu}}\left\vert y_{i}\right\vert ^{-\mu}\leq
16\frac{1}{c_{0}^{1-\mu}}Q^{-\mu}.
\end{align*}
Now, by Lemma \ref{lem:B(x):best}, $z$ is a best approximation vector of $\theta$, hence%
\[
d(\{\theta, \ldots ,Q\theta\},\Z^{2})\leq 2d(\{\widehat{z}, \ldots ,Q\widehat
{z}\},\Z^{2})\leq 32\frac{1}{c_{0}^{1-\mu}}Q^{-\mu}, %
\]
which is equal to $Q^{-\mu}$ by the choice of $c_{0}$.

It remains to see why $\mathbb{Z\theta+Z}^{2}$ is everywhere dense. This
simply follows from the fact that $e(\mathcal{H)}\leq\lambda_{2}(y)$ and that
$\lambda_{2}(y)$ tends to $0$ for $y\in E_{1}(x_{n})$ when $n$ goes to
infinity. 
\end{proof}

\subsection{A distorted tiling associated with the set of $\widehat{y}$, for $y$ in $E_{1}(x)$.}

Let $x$ be in $\Q_{\sigma}$ and let $u_{1},\ u_{2}$ be the reduced basis of $\Lambda_{x}$. 
For each $y$ in $E_1(x)$ there are unique integers $m$ and $a$, and $0\leq r<1$ 
such that $y=mu_1+u_2+(a+r)x$. For given integers $m$ and $a$, 
consider the trapezoid $T(m,a)$
with extreme points
\[
\widehat{x}+\frac{mu_{1}+u_{2}}{a\left\vert x\right\vert },\ \widehat{x}%
+\frac{(m+1)u_{1}+u_{2}}{a\left\vert x\right\vert },\ \widehat{x}+\frac
{mu_{1}+u_{2}}{(a+1)\left\vert x\right\vert },\ \widehat{x}+\frac
{(m+1)u_{1}+u_{2}}{(a+1)\left\vert x\right\vert }.
\]
For $y=mu_1+u_2+(a+r)x$ in $E_1(x)$ we set
\[
\mathcal{R}_y=\mathcal{R}_{\widehat{y}}=T(m,a).
\]
Let $\mathcal E(x)$ denote the set of $\widehat{y}$, $y$ in $E_1(x)$.
\begin{lemma}
\label{lem:distortedtiling}
There exists a constant $C_0$ such that for all $x$ in $Q_{\sigma}$ 
with $\left\vert  x \right\vert$ large enough and $c_2$ large enough, 
the collection $\mathcal{R}_{\widehat{y}}$, $\widehat{y}$ in $\mathcal{E}(x)$, 
is a $C_0$-distorted $H\times V$-tiling of $B(x)$ with
\begin{align*}
H=&\left\vert x \right\vert^h=\left\vert x \right\vert^{-\frac{(2-\mu)(b+\mu)}{(1-\mu)(b+1)}}\\
V=&\left\vert x \right\vert^v=\left\vert x \right\vert^{-\frac{(1+\mu)(b+\mu)}{(1-\mu)(b+1)}}.
\end{align*}
Moreover, every $\widehat{y}$ in $\mathcal E(x)$ lies on 
a vertical side of $\mathcal R_y$ and the minimal distance $\rho(x)$ 
between two elements in $\mathcal E(x)$ 
is $\gg V\asymp  \frac{\lambda_1(y)}{\left\vert  y \right\vert}$.  
\end{lemma}

\begin{proof}
Observe that the trapezoid $T(m,a)$ lies between the vertical lines
\[
V_m=\widehat{x}+\R(mu_1+ u_{2}),\text{ and } V_{m+1}=\widehat{x}%
+\R((m+1)u_1+u_{2})
\]
and between the horizontal lines 
\[
H_a=\widehat{x}+\frac{u_{2}}{a\left\vert x\right\vert }+\R u_{1},\text{ and }
H_{a+1}=\widehat{x}+\frac{u_{2}}{(a+1)\left\vert x\right\vert }+\R u_{1}. 
\]
Therefore, the trapezoids $\mathcal R_{\widehat{y}}$ have intersections
of Lebesgue measure zero.
Observe that for $y$ in $E_1(x)$, by Lemma \ref{lem:height},
$a\asymp \frac{\left\vert  y \right\vert}{\left\vert  x \right\vert}\asymp\left\vert x\right\vert ^{e_y-1}$. 
On the one hand, by definition of $E_1(x)$ and Lemma \ref{lem:y}, 
the distance between two consecutive horizontal lines is
\[
\asymp \frac{\lambda_{2}(x)}{a^{2}\left\vert x\right\vert} 
\asymp \frac{\left\vert x\right\vert\lambda_{2}(x)}{\left\vert y\right\vert^2}
\asymp \frac{\left\vert x\wedge y \right\vert}{\left\vert y\right\vert^2}
=\frac{\lambda_1(y)}{\left\vert  y \right\vert}
\asymp\left\vert y\right\vert ^{-1-\mu}
\asymp\left\vert x\right\vert ^{v}.
\]
 On the other hand, the distance between the two vertical segments of $T(m,a)$ is 
 $\asymp \frac{\lambda_1(x)}{\left\vert  y \right\vert}$ which is 
 $\asymp\left\vert  x \right\vert^h$ by Lemmas \ref{lem:height} and \ref{lem:xz}.
Since $h>v$, $\diam \mathcal R_{\widehat{y}}\asymp H$ and since $h<r_0$, we see that 
all the trapezoids $\mathcal R_{\widehat{y}}$ are included in $B(x)$ when $c_2$ is large enough.

Since $\widehat{y}=\widehat{x}+\frac{mu_{1}+u_{2}}{(a+r)\left\vert x\right\vert }$, 
$\widehat{y}$ is in the left vertical side of $\mathcal R_{\widehat{y}}$ and  
the nearest element of $\mathcal{E}(x)$ is in the same vertical line at a 
distance $\asymp V$. Therefore $\rho(x)\asymp \frac{\lambda_1(y)}{\left\vert  y \right\vert}$. 
\end{proof}

\subsubsection{Choice of the constants $c,\,c_1$ and $c_2$} 
The constant $c_2$ is chosen according to Lemma  \ref{lem:distortedtiling}. 
With this choice, we determine the constant $c_1$.
Since $\rho(x)\asymp \frac{\lambda_1(y)}{\left\vert  y \right\vert}$,  
it is possible to take $c_1$  small enough
 in order that for all $z$ in $D_1(y)$, 
 \[
 d(\widehat{y},\widehat{z})\leq  c_1 \frac{\lambda_{1}(y)}{\left\vert y\right\vert } \leq\frac14  \rho(x).
 \]

The choice of the constant $c$ involved in the definition of $Q_{\sigma}$ is done 
at the very end taking all the ``$\left\vert  x \right\vert$ large enough" into account.

\subsubsection{Distance between the points $\widehat{z}$ for $z$ in $D_{1}(y)$}

\begin{lemma}
\label{lem:d(z,z')} Let $x$ be in
$Q_{\sigma}$ and $y$ be in $E_{1}(x)$. If $z\neq z^{\prime}$ are in $D_{1}(y)$, 
then%
\[
d(\widehat{z},\widehat{z^{\prime}})\geq\frac{\lambda_{1}(y)}{2\left\vert
y\right\vert ^{1+2b}}\geq3\max_{u\in\sigma(x)}\operatorname{diam}B(u)
\]
when $\left\vert x\right\vert $ is large enough. 
Hence the balls $B(z)$, $z\in \sigma(x)$, are disjoint.
\end{lemma}

\begin{proof}
Choose a generator $u_{y}$ of $\Lambda_{y}^{\prime}%
=\Lambda_{y}\cap H_{y}$ and $y^{\prime}$ in $\Z^{3}\cap H_{y}$ such
that $\pi_{y}(y^{\prime})=u_{y}$ and $\left\vert y^{\prime}\right\vert
\leq\frac{1}{2}\left\vert y\right\vert $. We have $y^{\prime}=u_{y}+ry$ with
$\left\vert r\right\vert \leq\frac{1}{2}$ and $\Z^{3}\cap
H_{y}=\Z y+\Z y^{\prime}$.

Let $z=ay^{\prime}+ky$ be in $Q\cap H_{y}$. We have $z=au_{y}+(ar+k)y$, hence%
\[
\widehat{z}=\frac{a}{(ar+k)}\frac{u_{y}}{\left\vert y\right\vert }+\widehat
{y},
\]
and, since $z$ is primitive, the pair $(a,k)$ is primitive in $\Z^{2}$. 

Now, if $z=ay^{\prime}+ky$ is in $D_{1}(y)$ then
\[
\left\vert a\right\vert \lambda_{1}(y)=\left\Vert au_{y}\right\Vert_e
=\left\Vert \pi_{y}(z)\right\Vert_e \leq \frac14 \frac{\lambda_{1}(y)}{\left\vert
y\right\vert }\times\left\vert z\right\vert,
\]
hence $\left\vert a\right\vert \leq\frac{1}{4}\frac{\left\vert z\right\vert
}{\left\vert y\right\vert }$. Moreover, $\left\vert z\right\vert
\leq\left\vert k\right\vert \left\vert y\right\vert +\frac{1}{2}\left\vert
a\right\vert \left\vert y\right\vert $, thus
\[
\left\vert k\right\vert \geq\frac{1}{\left\vert y\right\vert }(\left\vert
z\right\vert -\frac{1}{2}\left\vert a\right\vert \left\vert y\right\vert
)=\frac{\left\vert z\right\vert }{\left\vert y\right\vert }(1-\frac{1}{8}%
)\geq\frac{1}{2}\frac{\left\vert z\right\vert }{\left\vert y\right\vert }.
\]

Let $z=ay^{\prime}+ky$ and $z^{\prime}=a^{\prime}y^{\prime}+k^{\prime}y$ be
two distinct points in $D_{1}(y)$. We have
\begin{align*}
\widehat{z}-\widehat{z^{\prime}} &  =(\frac{a}{(ar+k)}-\frac{a^{\prime}%
}{(a^{\prime}r+k^{\prime})})\frac{u_{y}}{\left\vert y\right\vert }\\
&  =(\frac{ak^{\prime}-a^{\prime}k}{(ar+k)(a^{\prime}r+k^{\prime})}%
)\frac{u_{y}}{\left\vert y\right\vert }.
\end{align*}
Since $(a,k)$ and $(a^{\prime},k^{\prime})$ are primitive, we have $ak^{\prime
}-a^{\prime}k\neq 0$. It follows that%
\[
d(\widehat{z},\widehat{z^{\prime}})\geq\frac{\lambda_{1}(y)}{\left\vert
y\right\vert }\times\frac{1}{\frac{\left\vert z\right\vert }{\left\vert
y\right\vert }\times\frac{\left\vert z^{\prime}\right\vert }{\left\vert
y\right\vert }}\geq\frac{\lambda_{1}(y)}{\left\vert y\right\vert ^{1+2b}}.
\]

It remains to see that for all $z$,
\[
\diam B(z)\leq\frac{1}{3}\frac{\lambda_{1}(y)}{\left\vert
y\right\vert ^{1+2b}}\asymp\frac{1}{\left\vert y\right\vert ^{1+2b+\mu}}.
\]
Since $\diam B(z)\asymp\left\vert z\right\vert ^{r_{0}}%
\asymp\left\vert y\right\vert ^{(1+b)r_{0}}$ and since%
\[
-(1+b)r_{0}-(1+2b+\mu)=\frac{(b+\mu)(2\mu-1)}{1-\mu}>0
\]
we have $\diam B(z)\leq\frac{1}{3}\frac{\lambda_{1}(y)}{\left\vert
y\right\vert ^{1+2b}}$.

 The last thing to see is that the balls $B(z)$, $z\in \sigma(x)$ are disjoint. 
 Recall that the constant $c_1$ has been chosen in order that for all $z$ in $D_1(y)$,
 $d(\widehat{y},\widehat{z})\leq\frac14  \rho(x)$, the minimal distance between two points 
 $\widehat{y}$, where $y$ is in $E_1$. It follows that the balls $B(z)$ are disjoint provided that
\[
\max_{z \in \sigma(x)} \diam B(z)< \frac14 \rho(x). 
\] 
This latter inequality holds because 
$\rho(x)\asymp \frac{\lambda_1(y)}{\left\vert  y \right\vert}$ and $\frac{\lambda_{1}(y)}{\left\vert
y\right\vert ^{1+2b}}\geq \diam B(z)$.
\end{proof}

\subsubsection{Number of points in $\sigma(x)$}

\begin{lemma}\label{lem:card:sigma}
Let $x$ be in $Q_{\sigma}$ and $y$ in $E_1(x)$. Then 
\begin{align*}
\card D_{1}(y)&\asymp\left\vert x\right\vert ^{2be_y}=\left\vert x\right\vert ^{d_1},\\
\card E_{1}(x)&\asymp\left\vert x\right\vert
^{2\frac{\mu-1}{1+b}+e_{y}}=\left\vert x\right\vert ^{e_1},
\end{align*}
and
\[
\card \sigma(x)\asymp\left\vert  x \right\vert^{n_x},
\]
where
\[
n_x=\frac{1}{\left(  1-\mu\right)  \left(  b+1\right)  }\left(  2b^{2}%
+2b\mu+b+\left(  2\mu-1\right)  \left(  2-\mu\right)  \right)
\]
when $\left\vert  x \right\vert$ is large enough.
\end{lemma}
\begin{proof}
It is not difficult to see that the number of points $z$ in $H_{y}\cap\Z^{3}$
such that 
$$
\left\Vert \widehat{y}-\widehat{z}\right\Vert_e \leq c_1 
\frac{\lambda_{1}(y)}{\left\vert y\right\vert }
\quad \hbox{and} \quad  
\frac{1}{2}\left\vert
y\right\vert ^{b}\leq\frac{\left\vert z\right\vert }{\left\vert y\right\vert
}\leq\left\vert y\right\vert ^{b}
$$
is $\asymp\left\vert y\right\vert ^{2b}$.
Indeed, the condition $\left\Vert \widehat{y}-\widehat{z}\right\Vert_e \leq
c_1 \frac{\lambda_{1}(y)}{\left\vert y\right\vert }$ is equivalent to
$\left\Vert \pi_{y}(z)\right\Vert_e \leq c_1\frac{\left\vert z\right\vert
}{\left\vert y\right\vert }\lambda_{1}(y)$ which means that there are
$\asymp\frac{\left\vert z\right\vert }{\left\vert y\right\vert }\asymp \left\vert y\right\vert^b $ 
possible values
for $\alpha=\pi_{y}(z)$. Moreover, the set of integer points on each of the lines
$\pi_{y}^{-1}(\alpha)$ is a translate of $\Z y$ and therefore 
there are $\left\vert y\right\vert^b$ possible $z$ for each $\alpha$. The fact that many of
such $z$ are primitive is less clear; in fact, by Lemma 7.11 of \cite{cheche}, we
have
\[
\card D_{1}(y)\asymp\left\vert y\right\vert ^{2b}\asymp
\left\vert x\right\vert ^{2be_{y}}.
\]
Clearly, 
\[
\card E_{1}(x)\asymp\frac{\lambda_{2}(x)}{\lambda_{1}(x)}%
\times\frac{\left\vert y\right\vert }{\left\vert x\right\vert }\asymp
\lambda_{2}^{2}(x)\left\vert x\right\vert^{e_y} \asymp\left\vert x\right\vert
^{2\frac{\mu-1}{1+b}+e_{y}}.
\]
It then follows that the number of points in $\sigma(x)$ satisfies
\[
\card \sigma(x)\asymp \left\vert x\right\vert ^{2be_{y}}\times   
\left\vert x\right\vert
^{2\frac{\mu-1}{1+b}+e_{y}}=\left\vert x\right\vert ^{n_{x}},
\]
where$\allowbreak$
\begin{align*}
n_{x}  &  =2\frac{\mu-1}{1+b}+(1+2b)\frac{\mu+b}{(1-\mu)(1+b)}\\
&  =\frac{1}{\left(  1-\mu\right)  \left(  b+1\right)  }\left(  2b^{2}%
+2b\mu+b+\left(  2\mu-1\right)  \left(  2-\mu\right)  \right).
\end{align*}
\end{proof}

\section{Lower bounds for the Hausdorff dimension: proofs}

\begin{proof}[Proof of Theorem \ref{thm:lower-general}]

Let $s$ be a positive real number. Suppose that the following conditions hold: 

\begin{itemize}
\item $\sum_{z\in\sigma(x)}(\diam B(z))^{s}\geq
(\diam B(x))^{s}$,

\item $R_{1}^{s}\card E_{1}(x)\gg(\diam B(x))^{s}$
where $R_{1}=\max_{y\in E_1(x)} \frac{\lambda_{1}(y)}{\left\vert y\right\vert }\asymp 
\left\vert x\right\vert ^{-e_{y}(\mu+1)}\asymp \left\vert x\right\vert ^{r_1}$,
\end{itemize}

\noindent for all $x$ in $Q_{\sigma}$ with $\left\vert x\right\vert $ large enough. 

Let us show that such an $s$ is a lower bound for $\dim_H\SSing^{\ast}(\mu)$. 
We want to use Theorem \ref{thm:massdistribution} with 
the self-similar structure $(Q_{\sigma},\sigma,B) $ which is strictly nested
and covers a subset of $\SSing^{\ast}(\mu)$ by Proposition
\ref{prop:SSS}. The first condition above is just the first
hypothesis of Theorem \ref{thm:massdistribution} and the second
condition of Theorem \ref{thm:massdistribution} is implied by Lemma
\ref{lem:d(z,z')}.   So it remains to check the last hypothesis of Theorem
\ref{thm:massdistribution}. For this last condition, we use
Lemma \ref{lem:counting} with $R_{0} =c_{2}\left\vert x\right\vert ^{r_{0}}$ 
and the sets $\mathcal{B}=B(\widehat{x},R_{0})$, $\mathcal{E}=\widehat
{E}_{1}(x)$, $\mathcal{D}_{\widehat{y}}=\widehat{D}_{1}(y)$, 
$\mathcal S=\cup_{\widehat{y}\in\mathcal{E}}\mathcal{D}_{\widehat{y}}$, $R_{1}$ defined
above, and
\[
R_{2}=c_{3}\left\vert x\right\vert ^{-e_{y}(\mu+1+2b)}\asymp\left\vert x\right\vert ^{r_2},\ R_{3}
=\left\vert
x\right\vert ^{e_{z}r_{0}}=\left\vert x\right\vert ^{\frac{\mu+b}{(1-\mu
)}r_{0}}= \left\vert x\right\vert ^{r_3}.
\]
Let us first check the inequalities between 
$R_0, \ldots ,R_3, H \asymp\left\vert x\right\vert ^{h}, V \asymp\left\vert x\right\vert ^{v}$. 
Looking at the exponents we find
\begin{align*}
r_0&=-\frac{  \mu^{2}%
-\mu+b+1}{\left(  1-\mu\right)  \left(  b+1\right)  } >
h=-\frac{(2-\mu)(b+\mu)}{(1-\mu)(b+1)}>v=r_1=-\frac{(1+\mu)(b+\mu)}{(1-\mu)(b+1)}>\\
r_2&=-\frac{(\mu+1+2b)(b+\mu)}{(1-\mu)(b+1)}
>r_3=-\frac{(b+\mu)}{(1-\mu)}\times \frac{  \mu^{2}%
-\mu+b+1}{\left(  1-\mu\right)  \left(  b+1\right)  },
\end{align*}
which show that the assumptions of Lemma \ref{lem:counting} about the numbers $R_0,\,R_1,\,R_2,\,R_3,\,H$ and $V$ are satisfied.
Moreover, by Lemma \ref{lem:d(z,z')} and for
$c_{3}$ small enough, we have
\[
d(z,z^{\prime})\geq \frac{\lambda_{1}(y)}{2\left\vert y\right\vert
^{1+2b}}\geq R_{2}, 
\]
for all $z\neq z^{\prime}$ in $D_{1}(y)$. Together with Lemma
 \ref{lem:distortedtiling}, this imply that all the assumptions of Lemma \ref{lem:counting} hold.

With the notations of Lemma \ref{lem:counting}, we get 
\[
f(r)\ll\max \Bigl\{ \frac{1}{R_{3}^{s}},\frac{R_{1}}{R_{2}%
R_{1}^{s}},\frac{R_{0}^{2}R_{1}}{VHR_{2}}\times\frac{1}{R_{0}^{s}} \Bigr\}.
\]
By Lemma \ref{lem:distortedtiling} (or \ref{lem:card:sigma}), $\frac{R_{0}^{2}}{VH}\asymp
\card E_{1}(x)$ and, since $\frac{R_{1}}{R_{2}}\asymp\left\vert
y\right\vert ^{2b}\asymp\card D_{1}(y)$, we see that
\[
\frac{R_{0}^{2}R_{1}}{VHR_{2}}\times\frac{1}{R_{0}^{s}}\asymp\frac
{\card \sigma(x)}{(\diam B(x))^{s}}. %
\]
With the first assumption $\sum_{z\in\sigma(x)}(\diam %
B(z))^{s}\geq(\diam B(x))^{s}$ we get
\[
\frac{1}{R_{3}^{s}}\ll\frac{R_{0}^{2}R_{1}}{VHR_{2}}\times\frac{1}{R_{0}^{s}}.
\]
With the second assumption $R_{1}^{s}\card E_{1}(x)\gg
(\diam B(x))^{s}$, we get%
\[
\frac{R_{1}}{R_{2}R_{1}^{s}}\asymp\frac{R_{0}^{2}}{VH}\frac{1}{R_{1}%
^{s}\card E_{1}(x)}\times\frac{R_{1}}{R_{2}}\ll\frac{R_{0}^{2}%
}{VH}\frac{1}{R_{0}^{s}}\times\frac{R_{1}}{R_{2}}.
\]
Therefore, for all $r$ in $\lbrack R_{3},R_{0} \rbrack$, we have
\[
\frac{\card \sigma(x)\cap B(a,r)}{r^{s}}\ll\frac
{\card \sigma(x)}{(\diam B(x))^{s}}%
\]
and so, with $F=B(a,r)$,%
\[
\sum_{z\in\sigma_{F}(x)}\frac{(\diam B(z))^{s}}%
{(\diam F)^{s}}\ll\sum_{z\in\sigma(x)}\frac
{(\diam B(z))^{s}}{(\diam B(x))^{s}}.
\]
By applying Theorem \ref{thm:massdistribution} we conclude that the Hausdorff
dimension of $\SSing^{\ast}(\mu)$ is at least equal to $s$.

The idea is now to show that the assumption $\sum_{z\in\sigma(x)}%
(\diam B(z))^{s}\geq(\diam B(x))^{s}$ is more restrictive than the
other assumption $R_{1}^{s}\card E_{1}(x)\gg (\diam B(x))^{s}$.

The condition
\[
sr_{3}+n_x>sr_{0}%
\]
implies the first assumption and the condition
\[
sr_{1}+e_{1}>sr_{0} 
\]
implies the second assumption. 
The first condition is equivalent to $s<\frac{n_x}{r_{0}-r_{3}}%
=s_{1}$ and the second is equivalent to $s<\frac{e_{1}}{r_{0}-r_{1}}=s_{2}$.
Therefore, to prove that $s_{1}$ is a lower bound for the Hausdorff dimension of
$\SSing^{\ast}(\mu)$, it is enough to check that $s_{1}<s_{2}$ for all
$\mu$ in $(\frac{1}{2},1)$ and all positive $b$.

Tedious calculations give%
\begin{align*}
s_{1} (b) &  =\frac{(1-\mu)\left(  2b^{2}+2b\mu+b+(2-\mu)(2\mu-1\right)
)}{\left(  b+2\mu-1\right)  \left(  \mu^{2}-\mu+b+1\right)  }, \\
s_{2} (b) &  =\frac{1}{2\mu+b\mu-1}\left(  -2\mu^{2}+5\mu+b-2\right),
\end{align*}
and%
\[
s_{2} (b) -s_{1} (b) =(b+\mu)\frac{\left(  2\mu^{2}-2\mu+1\right)  b^{2}+\left(
4\mu^{2}-2\mu\right)  b+\mu\left(  2-\mu\right)  \left(  2\mu-1\right)  ^{2}%
}{\left(  b-\mu+\mu^{2}+1\right)  \left(  b+2\mu-1\right)  \left(  2\mu
+b\mu-1\right),  }%
\]
which is $>0$ for $\mu$ in $[\frac{1}{2},1)$ and $b>0$. It follows that
\[
\dim_{H}\SSing^{\ast}(\mu)\geq s_{1} (b),%
\]
and the proof is complete. 
\end{proof}

\begin{proof}[Proof of corollaries \ref{cor:baker} and \ref{cor:4/3}]

Clearly,
\[
\lim_{b\rightarrow\infty}s_{1} (b) =2(1-\mu)
\]
It follows that
\[
\dim_{H}\SSing^{\ast}(\mu)\geq\lim_{b\rightarrow\infty}%
s_{1} (b) =2(1-\mu).
\]

Next we can compute the derivative of the function $b \mapsto s_1 (b)$. 
The numerator of this derivative is 
\[
Num(b)=(1-\mu)((2\mu^2-1)b^2+(8\mu^3-8\mu^2+2)b+(6\mu^4-7\mu^3+3\mu-1)).
\]
When $\mu>\frac{1}{\sqrt{2}}$,  $Num(b)$ is positive for all positive $b$, 
hence the maximum of $s_1$ is reached when $b$ goes to infinity. 
When $\mu<\frac{1}{\sqrt{2}}$, $Num(b)$ vanishes at the value 
\[
b_0=\frac{1}{1-2\mu^2}\left(\mu-4\mu^2+4\mu^3+\sqrt{(1-\mu)^3(2\mu-1)(2\mu-2\mu^2+1)}\right), 
\] 
which is positive. Since $Num(b)$ is negative for $b$ large this implies 
that $s_1 (b_0)$ is larger than the limit at infinity 
and therefore the Hausdorff dimension exceeds $2(1-\mu)$.\medskip

Let us look at the limit when $\mu$ tends to $\frac12$. With $b=\beta(2\mu-1)$, we obtain%
\[
s_{1}=s_{1}(\mu,\beta)=\frac{(1-\mu)(2\beta^{2}(2\mu-1)+2\beta\mu+\beta
+2-\mu)}{(\mu^{2}-\mu+1+\beta(2\mu-1))(\beta+1)}%
\]
Therefore for all $\beta>0$,%
\[
\lim_{\mu\rightarrow\frac{1}{2}}\dim_{H}\SSing^{\ast}(\mu)\geq
\frac{\frac{1}{2}(2\beta+\frac{3}{2})}{\frac{3}{4}(\beta+1)}. %
\]
Letting $\beta$ going to infinity this implies%
\[
\lim_{\mu\rightarrow\frac{1}{2}}\dim_{H}\SSing^{\ast}(\mu)\geq\frac
{4}{3}.
\]
\end{proof}

\begin{proof}[Proof of Proposition \ref{prop:packing}]
We keep the notation of the proof of Theorem \ref{thm:lower-general}. 
We want to use Lemma \ref{lem:self:packing} with the strictly 
nested self-similar structure $(Q_\sigma,\sigma,B)$ and the map $x\mapsto \widehat{x}$. 
We need to define the map $B'$. 
For $x$ in $Q_\sigma$, we set
\[
B'(x)=B(\widehat{x},c_{4}\left\vert x\right\vert ^{-\frac{\mu+1+2b}{1+b}}).
\]
Since, for $z\in \sigma(x)$, 
\[
\left\vert z\right\vert ^{-\frac{\mu+1+2b}{1+b}}\asymp \left\vert x \right\vert^{-e_y(1+\mu +2b)}
\asymp  \frac{\lambda_{1}(y)}{2\left\vert y\right\vert
^{1+2b}}\asymp R_2,
\]
making use of Lemma \ref{lem:d(z,z')}, we see that the balls $B'(z)$, $z\in \sigma(x)$, 
are disjoint when $c_4$ is small enough. Moreover, since $R_2$ is small compared to 
$\frac{\lambda_1(y)}{\left\vert y \right\vert}$,  
Lemma \ref{lem:nestedness} implies that for all $x\in Q_{\sigma}$, 
all $y\in E_1(x)$ and $z\in D_1(y)$, we have
\[
B'(z)\subset B(\widehat{y},\frac{2\lambda_1(y)}{\left\vert  y \right\vert})\subset B(x)\subset B'(x),
\]
hence the third assumption of Lemma \ref{lem:self:packing} holds. The second assumption 
of this lemma holds because $R_3$ is small compared to $R_2$. The fifth assumption, namely
\[
\sum_{z\in\sigma(x)}(\diam B'(z))^{s}\geq
(\diam B'(x))^{s},
\]
needs to be checked. Since 
\[
\diam B'(x) \asymp (\diam B(x))^{\frac{\mu+1+2b}{|r_0|(1+b)}}, 
\]
the fifth assumption holds provided that
\[
\frac{\mu+1+2b}{|r_0|(1+b)}s\leq s_1.
\]
Therefore,
\begin{align*}
\dim_P\SSing^{\ast}(\mu)
&\geq \frac{|r_0|(1+b)}{\mu+1+2b}s_1 \\
&=\frac{\left(  2b^{2}+2b\mu+b+(2-\mu)(2\mu-1\right)
	)}{(\mu+1+2b)\left(  b+2\mu-1\right)    }.
\end{align*}
Letting $b$ going to infinity, we obtain
\[
\dim_P\SSing^{\ast}(\mu)\geq 1.
\]
\end{proof}

\end{document}